\documentclass[]{article}
\catcode`\@=11
\catcode`\@=12
\usepackage{amsmath,xcolor}
\usepackage{amssymb}
\usepackage{amsthm}
\usepackage{enumitem}

\numberwithin{equation}{section}

\theoremstyle{plain}
\newtheorem{theorem}{Theorem}[section]

\newtheorem{corollary}[theorem]{Corollary}

\newtheorem{definition}[theorem]{Definition}

\newtheorem{lemma}[theorem]{Lemma}

\theoremstyle{definition}
\newtheorem{remark}[theorem]{Remark}

\newtheorem*{remark*}{Remark}
\newtheorem*{defn*}{Definition}

\newtheorem*{example*}{Example}
\newtheorem*{theorem*}{Theorem}
\newtheorem*{corollary*}{Corollary}
\newtheorem*{helein*}{H\'elein's Conjecture, $\bs n=\bf 3$}

\newcommand{\bs}{\boldsymbol}
\newcommand{\bb}{\mathbb}
\newcommand{\pd}{\partial}

\def\RR{\mathbb R}
\renewcommand{\SS}{\mathbb S^2}
\def\nn{\text{\mathbf n}}
\def\bn{\boldsymbol n}

\def\div{\mathop{\mathrm{dvi}}}

\def\nn{\mathbf n}

\newcommand{\ed}{\end{document}}

\begin{document}

\title{{A Proof of H\'elein's Conjecture on Boundedness of
Conformal Factors when n = 3}}
\author{P.I. Plotnikov \footnote{Lavrentyev Institute of
Hydrodynamics,
 Lavrentyev pr.~15,
Novosibirsk 630090, Russia, email: Plotnikov @ mail.ru}~ and J.F.
Toland
 \footnote{Department of Mathematical Sciences, University of Bath, Bath,
 BA2 7AY, UK, email:masjft@bath.ac.uk}}

\date{}

\maketitle

\tableofcontents

\begin{abstract}
 \emph{For smooth mappings of the
unit  disc into the oriented Grassmannian manifold $\mathbb
G_{n,2}$, H\'elein (2002)  conjectured the global existence of
Coulomb frames with bounded conformal factor provided the integral of $|\boldsymbol A|^2$, the squared-length of the
 second fundamental form, is less than $\gamma_n=8\pi$.
 It has since been shown   that  the optimal bounds on the integral of $|\bs A|^2$ that guarantee this result are:  $\gamma_3 = 8\pi$ and $\gamma_n = 4\pi$ for  $n \geq 4$.
For isothermal immersions, this hypothesis  is equivalent to saying the integral of the sum of the squares of the principal curvatures is less than $\gamma_n$.}

\emph{The goal here is to prove that when $n=3$ the same conclusion  holds under weaker hypotheses. In particular, it holds  for isothermal immersions when  $|\boldsymbol A|$ is  square-integrable and the integral of $|K|$, $K$ the Gauss curvature, is less than $4\pi$.  Since $2|K| \leq |\boldsymbol A|^2$ this implies the known result for isothermal immersions, but $|K|$ may be small  when $|\boldsymbol A|^2$ is  large.}

\emph{That  the result under the weaker  hypothesis is sharp is
shown by  Enneper's surface and stereographic projections. The method, which is purely analytic,    is  then extended to investigate the case when the length of the second fundamental
form is square-integrable.}

\end{abstract}

\section{Introduction}


In applications it is often important  to understand
the global behaviour of an immersed surface, or even a curve, in  $\RR^n$. Although  the theory of curves is non-trivial,
much has been done since the work of Axel Schur \cite{schur} (see
also Chern \cite[\S 4,\,\S 5]{chernsurfaces}) and recently the focus
has been on surfaces. For example, an important
question in conformal geometry in $\RR^n$ is   whether or not a surface has
 bi-Lipschitz isothermal coordinates. The following elementary  description of the case $n=3$ introduces notation and  sets the scene for what follows.

   \subsection{Isothermal Immersions}\label{isothermal} Let $D_1$ denote the closed unit disc centered at the origin in $\RR^2$. A smooth mapping
$\Psi: D_1\to \mathbb R^3$ is called an isothermal
immersion and $e^f$  is its
 conformal factor if,  with $
\partial_i={\partial}/{\partial X_i}
$,
\begin{equation}\label{aelita11-}
\partial_i\Psi(X)= e^{f(X)}\mathbf e_i(X), \quad \mathbf e_i(X)\cdot\mathbf
e_j(X)=\delta_{ij},\quad X \in D_1.
\end{equation}
Then the coefficients of the first fundamental form \cite[\S2.2]{kobayashi}, \cite[\S 6.1.1]{pressley} of the surface $\Psi(D_1)$ are  $E=G=e^{2f}$ and $F = 0$.  Since $\partial_{12}\Psi = \partial_{21}\Psi$ and $\partial_k(
\mathbf e_i\cdot\mathbf
e_j)=0$, $i,j,k \in \{1,2\}$, it follows that
 $$
    \partial_1 f=-\mathbf e_1\cdot\partial_2\mathbf e_2, \quad \partial_2
    f=\mathbf e_1\cdot\partial_1\mathbf e_2\text{~\rm in~} D_1,$$
    and hence
    \begin{equation}\label{Lapf} -\Delta f (X) = \partial_1 {\bf e}_1\cdot \partial_2 {\bf e} _2 -  \partial_1 {\bf e}_2 \cdot \partial_2 {\bf e}_1.
    \end{equation}
Let $\nn (X) = {\bf e}_1(X)\times {\bf e}_2(X)$.
Then, since     $\pd_1\Psi$ and $\pd_2\Psi$ are normal to $\nn$, the coefficients of the second fundamental form \cite[\S2.2]{kobayashi}, \cite[\S 7.1]{pressley}, of $\Psi(D_1)$ are
\begin{equation}\begin{split}&L := \pd_{11}\Psi\cdot \nn = -\pd_1\Psi\cdot \pd_1\nn;\quad N := \pd_{22}\Psi \cdot \nn = -\pd_{2}\Psi \cdot\pd_2 \nn,
\\&  - \pd_1\Psi\cdot \pd_2\nn =\pd_{21}\Psi\cdot \nn  =  :M := \pd_{12}\Psi\cdot \nn =
 -\pd_2\Psi\cdot \pd_1\nn,\end{split}\label{lef}
\end{equation}
 and, by \eqref{aelita11-},  the Gauss curvature $K$  of  $\Psi(D_1)\subset \RR^3$ is given by \cite[\S2.2]{kobayashi}, \cite[Cor. 8.1.3]{pressley},
\begin{align}\notag
K&= \frac{LN-M^2}{EG-F^2} = e^{-4f}(LN-M^2)\\ \notag &= e^{-4f}\big\{(\pd_1\Psi\cdot \pd_1\nn)(\pd_{2}\Psi \cdot\pd_2 \nn) -(\pd_1\Psi\cdot \pd_2\nn)(\pd_2\Psi\cdot \pd_1\nn)\big\}\\
\notag &=e^{-2f}\big\{({\bf e}_1\cdot \pd_1\nn)({\bf e}_2 \cdot\pd_2 \nn) -({\bf e}_1\cdot \pd_2\nn)({\bf e}_2\cdot \pd_1\nn)\big\}
\\\label{gausscurv} &=e^{-2f}\big\{(\pd_1{\bf e}_1\cdot \nn)(\pd_2{\bf e}_2 \cdot \nn) -(\pd_2{\bf e}_1\cdot \nn)(\pd_1{\bf e}_2\cdot \nn)\big\}.
\end{align}
Now   since ${\bf e}_i\cdot \nn= 0$, ${\bf e}_i\cdot{\bf e}_j =\delta_{ij}$ and $\|{\bf e}_j\|^2=1$ on $D_1$, $i,j= 1,2$,
\begin{align*}
\partial_j\mathbf e_1&=(\partial_j \mathbf e_1\cdot \mathbf
e_2)\mathbf e_2+(\partial_j \mathbf e_1\cdot \mathbf n) \mathbf n,
&\partial_j\mathbf e_1\cdot\mathbf e_2&=-\mathbf
e_1\cdot\partial_j\mathbf e_2,
\\
\partial_j\mathbf e_2&=(\partial_j \mathbf e_2\cdot \mathbf
e_1)\mathbf e_1+(\partial_j \mathbf e_2\cdot \mathbf n) \mathbf n,
 &\partial_j\mathbf e_i\cdot\mathbf n~&=-\mathbf
e_i\cdot\partial_j\mathbf n.
\end{align*}
Therefore,
$$
\partial_1\mathbf e_1\cdot \partial_2\mathbf e_2=(\mathbf n\cdot \partial_1\mathbf e_1)
(\mathbf n\cdot \partial_2\mathbf e_2)  \text{ and } \partial_2\mathbf e_1\cdot \partial_1\mathbf e_2=(\mathbf n\cdot \partial_2\mathbf e_1)
(\mathbf n\cdot \partial_1\mathbf e_2),
$$
and so, by  \eqref{Lapf}  and \eqref{gausscurv},
\begin{equation}\label{aelita11}  -\Delta f  = \partial_1\mathbf e_1\cdot \partial_2\mathbf e_2 - \partial_1\mathbf e_2\cdot \partial_2\mathbf e_1 = e^{2f}K.
\end{equation}
Note also that $\|\nn\|^2 = 1$ implies that
$\mathbf n\cdot\partial_i\mathbf n=0$, and therefore that
$\partial_i \mathbf n=-( \mathbf n\cdot \partial_i\mathbf e_1)\mathbf e_1 -(\mathbf n\cdot \partial_i\mathbf e_2)\mathbf e_2.$
Therefore
\begin{align*}\notag
 \partial_1\mathbf n \times \partial_2\mathbf n&=\big[(\mathbf n\cdot \partial_1\mathbf e_1)\mathbf e_1+(\mathbf n\cdot \partial_1\mathbf e_2)\mathbf e_2\big]
 \times \big[(\mathbf n\cdot \partial_2\mathbf e_1)\mathbf e_1+(\mathbf n\cdot \partial_2\mathbf e_2)\mathbf e_2\big] \\
&= \big((\mathbf n\cdot
\partial_1\mathbf e_1)(\mathbf n\cdot \partial_2\mathbf e_2)-
 (\mathbf n\cdot \partial_1\mathbf e_2)(\mathbf n\cdot \partial_2\mathbf e_1)\big) \mathbf n \notag \\
&=\big(
\partial_1\mathbf e_1\cdot\partial_2\mathbf e_2-
\partial_1\mathbf e_2\cdot\partial_2\mathbf e_1\big)\nn,
\end{align*}
and it follows that
\begin{equation*}
-\Delta f =   \nn \cdot (\partial_1\mathbf n \times \partial_2\mathbf n).
\end{equation*}

    Thus
\begin{equation}\label{aelita11+}
-\Delta f= \left\{ \begin{array}{l}K\, e^{2f}\\\Phi:= \nn \cdot (\partial_1\mathbf n \times \partial_2\mathbf n)\end{array}\right\} \quad  \text{~in~} D_1,
\end{equation}
where $K$ is the Gauss curvature of the surface $\Psi(D_1)$. To estimate the $L^2$-norm of $|\nabla f|$ it  suffices  to show
that the right-hand side of \eqref{aelita11+} is in $ W^{-1,2}$.

Now the isothermal immersion $\Psi$ of $D_1$ into
$\mathbb R^3$ has  an associated Gauss map, namely the  unit vector  $\nn(X),\,X \in D_1$, normal to the immersed surface.
Since $\|\nn\|^2 =1$ on $D_1$, it follows from \eqref{lef} that
\begin{align*}\pd_1 \nn &= (\pd_1 \nn\cdot {\bf e}_1) {\bf e}_1 + (\pd_1\nn \cdot {\bf e}_2){\bf e}_2 = -e^{-f}L{\bf e}_1 -e^{-f}M{\bf e}_2,
\\ \pd_2 \nn &= (\pd_2 \nn\cdot {\bf e}_1) {\bf e}_1 + (\pd_2\nn \cdot {\bf e}_2){\bf e}_2 = -e^{-f}M{\bf e}_1  -e^{-f}N{\bf e}_2,
\end{align*}
and hence
\begin{equation}\label{conc1}|\nabla \nn|^2 = e^{-2f}\big(L^2 + 2M^2 + N^2\big) = e^{2f}|\bs A|^2 \text{ on }D_1,
  \end{equation}
  where
  $$|\bs A|^2 =e^{-4f}\big( L^2 + 2M^2 + N^2\big) = (2H)^2- 2K$$
  and $H$, the mean curvature of $\Psi(D_1)$, is given by \cite[\S2.2]{kobayashi}, \cite[Cor. 8.1.3]{pressley},
  $$H = \frac{LG-2MF+NE}{2(EG-F^2)} =\frac{L+N}{2e^{2f}}.
  $$
Since $2H = \gamma_1 + \gamma_2$ and $K = \gamma_1\gamma_2$, where $\gamma_1,\gamma_2$ are the principal curvatures of the surface $\Psi(D_1)$, it follows that  $|\bs A|^2 = \gamma_1^2 +\gamma_2^2$. Here  $|\bs A|^2$, which is referred to as the squared-length of the second fundamental form of the surface,  is independent of the parametrization $\Psi$ of $\Psi(D_1)$. (For a general surface,  $|\bs A|^2 $ is the square of the Hilbert-Schmidt norm \cite[Ch. XI]{dunford} of its Weingarten map \cite[\S\S 7.2 \& 8.1]{pressley} on the tangent space.)  Moreover,
\begin{equation*}
\pd_1 \nn \times \pd_2\nn = e^{-2f}(LN-M^2)\nn = Ke^{2f} \nn.
\end{equation*}
This shows  for isothermal imbedding that
\begin{equation*}
\begin{split}
\int_{D_1}|\nabla \nn|^2 dX &=\int_{D_1} |\mathbf A|^2 e^{2f}\, dX=
\int_{D_1} |\mathbf A|^2 d\mu_g\\&= \int_{\Psi(D_1)} |\mathbf A|^2\, dS
= \int_{\Psi(D_1)} (\gamma_1^2 +\gamma_2^2)\, dS,\end{split}\end{equation*}
where $\mu_g = e^{2f} dX$, and
\begin{equation*}
\begin{split}
\int_{D_1}|\partial_1 \mathbf n\times
\partial_2\mathbf n| dX&=\int_{D_1} |K| e^{2f} \, dX= \int_{D_1}
|K| d\mu_g\\&= \int_{\Psi(D_1)} |K|\, dS = \int_{\Psi(D_1)}| \gamma_1\gamma_2|\, dS\end{split}
\end{equation*}
where $\gamma_1,\gamma_2$ are the principal curvatures of $\Psi(D_1)$. In general,
\begin{equation}\label{eqq}
2\int_{D_1}|\partial_1 \mathbf n\times
\partial_2\mathbf n| dX \leq  \int_{D_1}|\nabla \nn|^2 dX
\end{equation}
and equality holds for zero-mean curvature (minimal) surfaces.
\qed

   Toro  \cite{toro0,toro} used methods from geometric measure (variform) theory and harmonic analysis to  prove striking results that have led to significant developments. Among other things she proved the following.
\begin{theorem*} \textbf{(Toro)}
\emph{There exists $\varepsilon_0>0,$ such that any two-dimensional
surface $S$ in $ \mathbb R^n$, $n\geq 3$, has an isothermal
bi-Lipshitz parameterization for which the logarithm  of the
conformal factor is uniformly bounded if
\begin{itemize}  \setlength\itemsep{0em}
\item[($i$)]$S$ can be approximated in the Hausdorff metric by a
sequence of smooth surfaces $S_k$ for which there exists $\rho>0$,
$\beta>0$ such that the area of the intersection $S_k\cap B(x,
\rho)$ is less than $\beta$ for all $x\in S_k$ and all $k$.
\item[($ii$)] For $x\in S_k$,
\begin{equation}\label{aelita2}
    \int_{S_k\cap B(x,\rho)} |\mathbf A_k|^2\, dS <\varepsilon_0,
\end{equation}
where $|\mathbf A_k|$ is the length of the second fundamental form
of $S_k$.
\end{itemize}}
\end{theorem*}

\begin{remark}\label{eps*} Obviously an estimate of the optimal value of $\varepsilon_0>0$   is essential  if results like this are to be useful in applications, see Remark \ref{epsilon0}. \qed
\end{remark}
Motivated by  Toro's work ,  M\"uller \& S\u{v}er\'ak
\cite{sverakmuller} investigated properties of immersions of the
plane $\mathbb R^2$ into  Euclidean space $\mathbb R^n$ when the
second fundamental form is square-integrable. To do so they
 reformulated the problem in terms of the oriented
Grassmannian manifold $\bb G_{n,2}$, of two-dimensional oriented subspaces of $\mathbb R^n$, embedded in complex projective
space $\bb C\bb P^{n-1}$, and  used
compensated-compactness methods from  partial-differential-equations
theory.

\subsection{H\'elein's Conjecture $\bs n=3$}

In his monograph, H\'elein proved a result  \cite[Lem.\,5.1.4]{helein},  on
mappings from the unit disc
$D_1$ in $\RR^2$ into the Grassmannian manifold $\mathbb G_{n,2}$, which is widely
used when analyzing  variational problems in the theory of surfaces
with bounded Willmore energy.  To do so he  did
not assume   that   $u$, from the disc to the Grassmannian manifold,
corresponds to  oriented tangent spaces  to a surface.

\begin{theorem*} \textbf{(H\'elein   \cite[Lem.\,5.1.4]{helein})}
\emph{For a  mapping $u: D_1 \to \mathbb G_{n,2}$  with
\begin{equation}\label{helein:}
\|d u\|^2_{L^2(D_1)}\leq  8\pi/3-\delta , \quad \delta>0,
\end{equation}
there exists $ (\mathbf e_1(X), \mathbf e_2(X))\in u(X),~X \in D_1$,
such that
$$
\mathbf e_i\cdot \mathbf e_j=\delta_{ij}, \quad \|\nabla\mathbf
e_j\|_{L^2(D_1)}\leq c(\delta),
$$
and the frame $\mathbf e_i$ forms a so-called Coulomb frame.}
\end{theorem*}
He further conjectured \cite[Conj. 5.2.3]{helein} that the same conclusion should holds when $8\pi/3$ in \eqref{helein:} is replaced by $8\pi$.
For $n=3$,
$\mathbb G_{3,2}$ can be identified with the  sphere $\mathbb S^2$
of unit vectors in $\mathbb R^3$, and $u$ in \eqref{helein:} with $\mathbf n:D_1 \to
\mathbb S^2$. In that  case H\'elein's conjecture has the following form in which it is not assumed that  $\mathbf n$ is
a field of normals to a surface in $\RR^3$.

\begin{helein*}\label{inna2} Let  $\mathbf n: D_1\to \mathbb S^2$
satisfy
\begin{equation}\label{inna1}
    \int_{D_1} |\nabla \mathbf n|^2\, dX\leq 8\pi-\delta, \quad \delta >0.
\end{equation}
Then there exist an orthonormal moving frame $(\mathbf e_1, \mathbf e_2)$
and a function $f$ such that the vectors
$\mathbf e_i(X)$ are orthogonal to $\mathbf n(X)$, the
frame $(\mathbf e_1(X), \mathbf e_2(X), \mathbf n(X))$ has
positive orientation, $\mathbf n\cdot (\mathbf e_1\times \mathbf e_2)>0$, and
\begin{equation}\label{inna3}
\|\mathbf e_i\|_{W^{1,2}(D_1)}\leq c(\delta), \quad i=1,2,
\end{equation}
\begin{equation}\label{inna4}
    \partial_1 f=-\mathbf e_1\cdot\partial_2\mathbf e_2, \quad \partial_2
    f=\mathbf e_1\cdot\partial_1\mathbf e_2\text{~\rm in~} D_1,
\end{equation}
\begin{equation}\label{inna5}
    -\Delta f=\partial_1 \mathbf e_1\cdot \partial_2\mathbf e_2-
    \partial_2 \mathbf e_1\cdot \partial_1\mathbf e_2
    \text{~\rm in~} D_1, \quad  f=0\text{\rm~on~}\partial D_1,
\end{equation}
\begin{equation}\label{inna6}\begin{split}
 \|\nabla f\|_{L^2(D_1)}\leq c(\delta), \quad
   \|f\|_{L^\infty(D_1)}\leq c(\delta).
\end{split} \end{equation}
\end{helein*}

Following a Grassmannian approach, Sch\"atzle \cite[Appendix A, Prop. 5.1]{schatzle}  confirmed the conjecture for all $n$ when he proved the following.

\theoremstyle{plain}
\newtheorem*{Sch*}{Proposition(Sch\"atze \cite[Prop. 5.1]{schatzle} )}
\begin{Sch*}

Let $f : D_1 \subset \RR^2 \to \RR^n$  be a conformal immersion with induced
metric $g_{ij} = e^{2u}\delta_{ij}$ and square integrable second fundamental  form $\bs A$ satisfying
\begin{equation}\label{aelita11} \int_{D_1} |\bs A|^2d\mu_g \leq \left\{ \begin{array}{l}8\pi-\delta \text{ for } n=3,\\ 4\pi-\delta  \text{ for } n\geq 4,\end{array}\right.\end{equation}
for some $\delta >0$
Then there exists a smooth solution $v : D_1 \to  \RR$ of
\begin{equation}\Delta_g v = K_g \text{ on }D_1 \label{aelita15a}\end{equation}
satisfying
$$
\|v\|_{L^\infty(D_1)},~ \|Dv\|_{L^2(D_1)},~ \|D^2v\|_{L^1(D_1)} \leq  C(n, \delta)  \int_{D_1} |\bs A|^2d\mu_g .
$$
\end{Sch*}

The proof is based on  Kuwert\&
Sch\"{a}tzle \cite{kuwertsch}, and on the M\"uller-S\u{v}er\'ak
estimates of the K\"ahler form in  complex  projective space.

\begin{remark}
 Suppose $\Psi: D_1\to \mathbb R^3$ is an isothermal
immersion with conformal factor $e^f$. Then
 $v-f$ is harmonic by \eqref{aelita11+} and  \eqref{aelita15a}.   It then   follows from \eqref{conc1} and \eqref{aelita11}  that H\'elein's Conjecture for $n=3$  is a Corollary of  Sch\"atzle's proposition above,  at least for immersions. That the result is optimal is shown by examples using
Enneper's surface following \cite[Cor. 5.1]{kuwert}.\qed
\end{remark}

\begin{remark}\label{epsilon0}
 By Toro's theorem, singularities occur only if energy in excess
of  $\varepsilon_0$  in \eqref{aelita2}   concentrates  at a finite
number of distinct points. So H\'elein's Conjecture for $n=3$
implies   $\varepsilon_0 \geq 8\pi$  (see Remark \ref{eps*}).\qed
\end{remark}

 {In what follows, conclusions \eqref{inna3} -\eqref{inna6} are established under the hypothesis
\begin{equation}
\mathbf n \in W^{1,2}(D_1, \mathbb S^2),\quad
\int_{D_1}|\partial_1 \mathbf n\times
\partial_2\mathbf n| dX \leq 4\pi - \delta,\quad \delta >0,\label{inna1+}
\end{equation}
which is implied by, but weaker than,   \eqref{inna1};
 see the discussion leading to \eqref{eqq}.}

\subsection{Applications}
Identifying the  best constant in H\'elein's Conjecture is
particularly important  when $n=3$  because  of
applications,  such as  arise when deciding the regularity of
solutions of variational problems in biology, elasticity theory, and
elsewhere when the unknown variables describe a surface. For example:

\emph{In biology,} the  main component of the Heimlich functional, in
the variational theory of problems involving living cell membranes,
is the Willmore conformal energy.   Hence  questions of regularity
of surfaces with square-integrable second fundamental forms  and the
possibility of  singularities are important in the
mathematical theory of biological membranes, see
\cite{mondino}, \cite{KMR} and references therein.\qed

\emph{In hydroelasticity,}  the surface shape of steady  waves on  a
three-dimensional expanse of fluid which is at rest at infinite
depth and moves irrotationally under gravity, bounded above by a
frictionless elastic sheet which has gravitational potential energy,
bending energy proportional to the square integral of its mean
curvature (its Willmore functional), and stretching energy
determined by the position of its particles relative to a reference
configuration, is governed by equations that arise as critical
points of a natural Lagrangian. The resulting theory
\cite{plottol_1,plottol_2} must ensure that the wave  surface is
non-self-intersecting. \qed

\emph{In general relativity,} when  the universe is defined as
$\mathbb R^3$ with a metric tensor $g_{ij}$, the so-called
quasi-local Hawking mass energy $m(\Sigma)$  is defined for  an
arbitrary domain $\Omega\subset \mathbb R^3$, bounded by a closed
surface $\Sigma$ with mean curvature $H$, as
\begin{equation*}
    m(\Sigma)= \frac{(\text{Area}\Sigma)^{1/2}}{(16\pi)^{3/2}}
    \Big(16\pi-\int_\Sigma H^2\, d\Sigma\Big).
\end{equation*}
See  \cite{KOE} for references and further reading. \qed

\subsection{Organization of the Paper}
In  the hope of being accessible to mathematicians  who need  these results for applications but  are not professional geometers,
this paper deals exclusively with the case  $n=3$ using only analysis techniques and
partial differential equations,  without reference to complex projective space or Grassmannian manifolds.

Section \ref{main} is a brief description of Theorem \ref{aelita22} and \ref{aelita30}, which are the main results upon which the rest of the
discussion relies.

Section \ref{asol} is  a self-contained yet elementary  proof of
Theorem \ref{aelita22}. The notation is set out in Section \ref{asol2},
estimates are developed in Section \ref{estimates}, and the main part of the proof is in Section
\ref{asol3}. Then, in Section \ref{inna},   {an extension of H\'elein's Conjecture for
$n=3$} is established using Theorem \ref{aelita22} and the continuation argument
in
 H\'elein's book \cite{helein}.

In  Section \ref{krit000}, the theory of Section \ref{asol} is
extended to uncover what can be said without the
hypothesis  of H\'elein's Conjecture  or, more precisely, without
hypothesis \eqref{aelita23}  in Theorem \ref{aelita22}. The proof of
Theorem \ref{aelita30}
 relies on constructions from Section \ref{asol}, augmented by   a corollary of
Federer's Theorem \cite[Thm. 3.2.22]{federer} from geometric measure
theory.

In  Section \ref{niza},
 Enneper's
classical minimal surfaces, which when suitably parameterized share their Gauss
maps with stereographic projections, provide examples that  show the sense in which the  results obtained are optimal.

\section{Notation and Main Results}\label{main}
\subsection{Notation}

Let
\begin{align*}
D_r&=\{|X|\leq r,\, X=(X_1,X_2)\in \RR^2\}, \text{ a closed disk  in
the plane},
\\  D_r^\circ &=\{|X|< r,\, X\in \RR^2\}, \text{ an open  disk  in the plane,}\\
\mathbb S^2 &= \{\boldsymbol \xi \in \mathbb R^3: |\boldsymbol \xi |=1\}, \text{ the unit sphere in $\RR^3$},\\
\boldsymbol k &= (0,0,1), \text{ the north pole of }\SS, \text{ and }\SS_0 = \SS \setminus \{+{\boldsymbol k},-{\boldsymbol k}\},\\
 W^{1,2}(D_r, \mathbb S^2) &= \Big\{\mathbf u \in  W^{1,2}(D_r, \RR^3):  \mathbf u(X) \in \SS \text{almost everywhere on }D_r\Big\}.
\end{align*}

\begin{defn*}
A map $\mathbf u \in W^{1,2}(D_1, \mathbb S^2)$, which has finite
energy
\begin{equation}\label{aelita18} E(\mathbf u):=   \int_{D_1}|\nabla\mathbf u|^2\, dX<\infty, \end{equation}
is said to be smooth, written $\mathbf u \in C^\infty(D_1,\SS)$, if
for some $r>1$ there is an infinitely differentiable  $\mathbf v:
D^\circ_r\to \SS$ with $\mathbf v (X) = \mathbf u(X)$ almost
everywhere on $D_1$. Let $ C^\infty_0(D_1,\SS) = \{\mathbf u \in
C^\infty(D_1,\SS): \mathbf u \text{ has compact support in }
D_1^\circ\}$ and let $W^{1,2}_0(D_1,\SS)$ be the completion of
$C^\infty_0(D_1,\SS)$ in $W^{1,2}(D_1,\SS)$.
\end{defn*}
\begin{lemma}\label{smooth}  ${C^\infty(D_1,\SS)}$ is dense in $W^{1,2}(D_1, \mathbb S^2)$.
\end{lemma}
\begin{proof} For $\mathbf u \in W^{1,2}(D_1, \mathbb S^2)$ and $r>1$, let $\mathbf u_r(X) = \mathbf u(X/r),\, X \in D_r.$
Then  $\mathbf u_r \in W^{1,2}(D_r, \mathbb S^2)$, its restriction,
$\hat{\mathbf  u}_r$, to $D_1$ is in $W^{1,2}(D_1, \mathbb S^2)$,
and $\hat {\mathbf u}_r \to\mathbf u$ in $W^{1,2}(D_1, \mathbb S^2)$
as $r \to 1$. Now from   \cite[\S 4]{schoenu} (see also \cite{betz})
it follows that there is a smooth function
$$\mathbf v_r \in C^\infty(D_r,\SS) \text{ such that  }\|\mathbf v_r-\mathbf u_r\|_{W^{1,2}(D_r, \mathbb S^2)} \leq r-1.$$
Hence $\|\hat{\mathbf v}_r-{\mathbf u}\|_{W^{1,2}(D_1, \mathbb S^2)}
\to 0 \text{ as } r \to 1$ and, since $\hat{\mathbf v}_r \in
C^\infty(D_1,\SS)$, the proof is complete.
\end{proof}
Now when $\nn \in W^{1,2}(D_1,\SS)$ put
\begin{equation}\label{aelita19}
    \Phi(X):=\mathbf n(X)\cdot(\partial_1\mathbf n(X)\times
    \partial_2\mathbf n(X)),
\end{equation}
 and note that $\Phi\in L^1(D_1,\SS)$ and
\begin{equation}\label{aelita20}
    \int_{D_1}|\Phi|\, dX \leq \frac{1}{2}
    \int_{D_1}|\nabla\mathbf n|^2\, dX.
\end{equation}
Since $\partial_i\mathbf n,\,i=1,2$, are orthogonal to the unit
vector $\mathbf n$ and the vector field $\partial_1\mathbf
n\times\partial_2\mathbf n$ is parallel to $\mathbf n$, it follows
that
$$
\Phi=\text{sign~}(\Phi)\, |\partial_1\mathbf
n\times\partial_2\mathbf n|.
$$
Thus the area on  $\SS$   of the image under $\nn$  of an area
element $dX$ of $D_1$ is
\begin{equation}\label{aelita21}dS=|\partial_1\mathbf
n\times\partial_2\mathbf n|\,dX, \text{ and hence
}\text{meas}\big(\mathbf n(D_1)\big)\leq     \int_{D_1}|\Phi|\, dX.
\end{equation}

\subsection{Main Results}
 A corollary (see Section \ref{inna}) of Theorem \ref{aelita22}, is that \eqref{inna3}-\eqref{inna6} in H\'elein's Conjecture hold when hypothesis \eqref{inna1} is replaced \eqref{aelita23}, which is weaker. 

 \begin{theorem}\label{aelita22}
Suppose $\mathbf n \in W^{1,2}(D_1, \mathbb S^2)$ satisfies
\begin{equation}\label{aelita23}
    \int_{D_1}|\partial_1 \mathbf n\times \partial_2 \mathbf n|\, dX\leq
    4\pi-\delta.
\end{equation}
Then  there exist $\Omega_i\in L^2(D_1)$ with
$\|\Omega_i\|_{L^2(D_1)}\leq (2^3\pi/\delta)\|\nabla \mathbf
n\|_{L^2(D_1)}$, $i = 1,2$, and for every $\zeta\in
C^\infty_0(D_1)$,
\begin{equation}\label{aelita24}
\int_{D_1} \Phi\,\zeta\,
dX=\int_{D_1}(\Omega_2\,\partial_1\zeta-\Omega_1\,\partial_2\zeta)\,
dX.
\end{equation}
 In particular, for $\delta$ in \eqref{aelita23} and an absolute constant $c$,
\begin{equation}\label{aelita25}
    \|\Phi\|_{W_0^{-1,2}(D_1)}\leq \frac{c}{\delta}\|\nabla\mathbf n\|_{L^2(D_1)}.
\end{equation}
\end{theorem}

\begin{remark}\label{aelita26} Since
$ |\Phi|\leq |\partial_1\mathbf n||\partial_2\mathbf n|\leq
\frac{1}{2}|\nabla\mathbf n|^2, $ condition \eqref{aelita23}  is
satisfied if \eqref{inna1}, the hypothesis of H\'elein's
Conjecture, holds, but not vice versa.\qed
\end{remark}

To investigate what can be said when there is no restriction on the
energy of $\mathbf n$ except that it is finite, let $\mathcal
A\subset \mathbb S^2$ and $  \mathcal F:=\nn^{-1}(\mathcal A) =
\big\{X\in D_1:\, \mathbf n(X)\in \mathcal A\big\}.$

\begin{theorem}\label{aelita30}
If $\mathbf n \in W^{1,2}(D_1, \mathbb S^2)$ and $\mathcal A\subset
\mathbb S^2$ is Borel with positive measure $\mu$,
 there exist $\Omega_i\in L^2(D_1),\,i=1,2$, such that
for all $\zeta\in C^\infty_0(D_1)$,
\begin{gather}\label{aelita31}
\int_{D_1} \Phi\,\zeta\, dX=\frac{4\pi}{\mu}\int_{\mathcal
F}\Phi\,\zeta\,
dX+\int_{D_1}(\Omega_2\,\partial_1\zeta-\Omega_1\,\partial_2\zeta)\,
dX,\\
\intertext{and $\|\Omega_i\|_{L^2(D_1)}\leq
{c}\,\mu^{-1/2}\,\|\nabla \mathbf n\|_{L^2(D_1)}$, where  $c$ is an
absolute constant.
 Thus}
\big\|\Phi-\frac{4\pi}{\mu}\chi_{_{\mathcal
F}}\Phi\big\|_{W^{-1,2}(D_1)} \leq \frac{c}{\mu^{1/2}}\,\|\nabla
\mathbf n\|_{L^2(D_1)}. \label{aelita32}
\end{gather}
\end{theorem}

\begin{remark}
If $\{\mathcal A_j:1\leq j \leq N\}$ is a family of mutually disjoint subsets of $\SS$, each with measure $ \mu_j$, the
 corresponding family $\{\mathcal F_j:1\leq j \leq N\}$ of their inverse images under $\nn$
 are  mutually disjoint in $D_1$ and, by \eqref{aelita31}, for each $j$ there exist $\Omega^j_i\in L^2(D_1),\,i=1,2$, such that, for all $\zeta\in C^\infty_0(D_1)$,
\begin{equation*} 4\pi \int_{\mathcal F_j}\Phi\,\zeta\,
dX =
\mu_j\left( \int_{D_1} \Phi\,\zeta\, dX - \int_{D_1}(\Omega^j_2\,\partial_1\zeta-\Omega^j_1\,\partial_2\zeta)\,
dX\right).
\end{equation*}
 Now let $\mu = \sum_{j=1}^N \mu_j$,  $\mathcal A = \cup_{j=1}^N \mathcal A_j$,         $\mathcal F = \cup_{j=1}^N \mathcal F_j$, and sum  over $j$ to obtain
\begin{align*}4\pi \int_{\mathcal F}\Phi\,\zeta\,
dX&= \mu\left( \int_{D_1} \Phi\,\zeta\, dX - \sum_{j=1}^N\left(\frac{\mu_j}{\mu}\int_{D_1}(\Omega^j_2\,\partial_1\zeta-\Omega^j_1\,\partial_2\zeta)\,
dX\right)\right) \\
&= \mu\left( \int_{D_1} \Phi\,\zeta\, dX - \int_{D_1}(\widetilde\Omega_2\,\partial_1\zeta-\widetilde \Omega_1\,\partial_2\zeta)\,
dX\right),
\end{align*}
where
$$
\widetilde\Omega_i = \sum_{j=1}^N\left(\frac{\mu_j}{\mu}\right)\, \Omega^j_i,\quad i= 1,2.
$$
Thus, $\widetilde \Omega_i$ satisfies \eqref{aelita31} when  $\mathcal A = \cup_{j=1}^N \mathcal A_j$ and $\mathcal A_j \cap \mathcal A_k = \emptyset$, $j \neq k$.
For example,
 $$\sum_{j=1}^N\mu_j\int_{D_1}(\Omega^j_2\,\partial_1\zeta-\Omega^j_1\,\partial_2\zeta)\,
dX  = 0 \text{  for all } \zeta\in C^\infty_0(D_1) \text{  if }\cup_{j=1}^N \mathcal A_j = \SS,$$ since $\mu = 4\pi$ and,  by \eqref{aelita31}, $$ \int_{D_1}(\Omega_2\,\partial_1\zeta-\Omega_1\,\partial_2\zeta)\,
dX = 0 \text{  for all $\zeta\in C^\infty_0(D_1)$ when $\mathcal A = \SS$}.\quad \Box$$
\end{remark}
\begin{remark}\label{holo}
Theorem \ref{aelita30} implies  that  the pre-image of every subset
$\mathcal A \subset\SS$  of positive measure, however small,
contains significant information about the singularity of $\Phi$
(reminiscent of holography, when any fragment, however small, of a
glass plate which contains a holographic image, contains the entire
image). To see that Theorem \ref{aelita30} is a generalisation of
Theorem \ref{aelita22}, let $\mathcal A = \SS\setminus \nn (D_1)$ so
that $\mathcal F = \emptyset$ and $\chi_{_{\mathcal F}} = 0$ on
$D_1$.\qed\end{remark}

Recall from  \eqref{aelita20} that  $\Phi$ is bounded in $L^1(D_1)$
in terms of the energy $E(\nn)$.
 However,   sequences $\{{\mathbf n}_k\}$  of vector fields,
 with $\{E ({\bf n}_k)\}$  bounded  but for which the
corresponding sequence  $\{\Phi_k\}$ is  not bounded in
$W^{-1,2}(D_1)$, are discussed in Section \ref{niza}.  The first
example to illustrate Theorems \ref{aelita22} and \ref{aelita30}
involves
 the Enneper hyperbolic surface for which
singularities arises as a result of branch point formation. The
second involves  stereographic projection onto a sphere, where
singularity are associated with   bubble formation. Curiously, the
two examples have the same Gauss map, and in that sense there is
only one example.

\section{H\'elein's Conjecture when $\boldsymbol{ n = 3}$}
\label{asol}

The idea  underlying the proof of Theorem \ref{aelita22}  is, for a
given $\nn \in W^{1,2}(D_1, \SS)$, to
 write $\Phi$ in  weak divergence form
\begin{equation}\label{asol1}
 \Phi=\partial_2 \omega_1-\partial_1\omega_2,
\end{equation}
and establish appropriate estimates.
 Since such a representation on the whole disc $D_1$ is not possible, even when the corresponding  $\nn$ is smooth because the resulting $\omega_i$ may have
strong singularities, the task is limited to showing that an
appropriate representation is possible under the hypotheses of
Theorem \ref{aelita22}.  Since, by Lemma \ref{smooth},  every
$\mathbf n \in W^{1,2}(D_1,\SS)$ can be approximated in
$W^{1,2}(D_1,\SS)$ by a sequence  $\{\mathbf n_k\}$ of vector fields in
$C^\infty(D_1,\SS)$, it suffices to prove the theorem for smooth
$\nn$ satisfying \eqref{aelita23}.

\subsection{Construction of $\boldsymbol \omega_i$}\label{asol2}

The representations of a vector field  $\nn:D_1 \to \SS$  by
Cartesian coordinates and
 spherical coordinates are related as follows:
\begin{equation*}\label{leila8}
    \mathbf n=(n_1,n_2,n_3),\quad n_1=\cos\varphi\sin
    \vartheta,\quad
    n_2=\sin\varphi\sin\vartheta, \quad  n_3=\cos\theta,
\end{equation*}
$\theta\in [0,\pi)$, and $\varphi\in (0,2\pi]$. Here $\nn=
(n_1,n_2,n_3), \vartheta, \varphi$ are functions of $X \in D_1$.
Then formal partial differentiation yields, the relation
coordinates
\begin{align*}
\partial_i \mathbf n&= \left(
\begin{array}{c}
\displaystyle { \cos \vartheta\cos\varphi}
\\ \displaystyle {\cos\vartheta\sin\varphi
 }   \\\displaystyle {-\sin \vartheta
 }
  \end{array}
\right)^\top\,\partial_i\vartheta+ \left(
\begin{array}{c}
\displaystyle {- \sin \vartheta\sin\varphi}
\\ \displaystyle {\sin\vartheta\cos\varphi
 }   \\\displaystyle {0 }
  \end{array}
\right)^\top\,\partial_i\varphi,
\\
\intertext{whence}
\partial_1 \mathbf n\times \partial_2\mathbf
n&=(\partial_1\vartheta\,\partial_2\varphi-
\partial_2\vartheta\,\partial_1\varphi)
 \left(
\begin{array}{c}
\displaystyle { \cos \vartheta\cos\varphi}
\\ \displaystyle {\cos\vartheta\sin\varphi
 }   \\\displaystyle {-\sin \vartheta
 }
  \end{array}
\right)^\top\times \left(
\begin{array}{c}
\displaystyle {- \sin \vartheta\sin\varphi}
\\ \displaystyle {\sin\vartheta\cos\varphi
 }   \\\displaystyle {0 }
  \end{array}
\right)^\top\\&~\\
&=(\partial_1\vartheta\,\partial_2\varphi-
\partial_2\vartheta\,\partial_1\varphi)
 \left|
\begin{array}{ccc}
\displaystyle { \boldsymbol i} &\displaystyle{\boldsymbol j} & \displaystyle{\boldsymbol k} \\
\displaystyle { \cos \vartheta\cos\varphi} &\displaystyle{
\cos\vartheta\sin\varphi} &\displaystyle{-\sin\vartheta}
\\ \displaystyle {-\sin\vartheta\sin\varphi} &
\displaystyle{\sin\vartheta\cos\varphi}&\displaystyle{0
 }
  \end{array}
\right|\\
&=\sin\vartheta\,(\partial_1\vartheta\,\partial_2\varphi-
\partial_2\vartheta\,\partial_1\varphi)\,
\mathbf n.
\end{align*}
Then since
\begin{align*}
\mathbf n\cdot(\partial_1\mathbf n\times
    \partial_2\mathbf n)&=\sin\vartheta\,(\partial_1\vartheta\,\partial_2\varphi-\partial_2\vartheta\,\partial_1\varphi)
\\
&=\partial_2(\cos\vartheta+1)\,\partial_1\varphi-
\partial_1(\cos\vartheta+1)\,\partial_2\varphi
\\&= \partial_2\big((\cos \vartheta +1)\partial_1 \varphi\big) - \partial_1\big((\cos \vartheta +1)\partial_2 \varphi\big),
\end{align*}
and
$$
\cos\vartheta + 1=n_3 + 1,\quad
\partial_i\varphi=\frac{n_1\partial_i
n_2-n_2\partial_i n_1}{n_1^2+n_2^2},
$$
there emerges a formula which apparently gives $\Phi$  in the
divergence form \eqref{asol1}
\begin{multline}\label{leila10}
\mathbf n\cdot(\partial_1\mathbf n\times
    \partial_2\mathbf n)=\partial_2\Big\{\frac{n_3 +1}{
n_1^2+n_2^2}(n_1\partial_1 n_2-n_2\partial_1 n_1)\Big\} \\-
\partial_1\Big\{\frac{n_3+1}{
n_1^2+n_2^2}(n_1\partial_2 n_2-n_2\partial_2 n_1)\Big\}\\=
\partial_2\omega_1 - \partial_1\omega_2, ~ \text{say}. \qquad\qquad\qquad\quad
\end{multline}
However, since $n_3 = \pm \sqrt{n_1^2 +n_2^2}$  on $\SS$,
\begin{equation}\label{leila10b}
(n_1,n_2,n_3) \mapsto \frac{n_3+1}{n_1^2+n_2^2}, \in \RR,\quad
(n_1,n_2,n_3) \in \SS,
\end{equation}
is real-analytic on $\SS$ except where
$(n_1,n_2,n_3)=(0,0,1)=:\boldsymbol k$. It follows that when $\nn:
D_1 \to \SS$ is smooth,  $\omega_i,i=1,2,$ in \eqref{leila10} is
smooth where $\mathbf n(X)\neq\boldsymbol k$, but there may be
singularities where  $\nn(X) =\boldsymbol k$. Thus \eqref{leila10}
may not hold in the sense of distributions on $D_1$ if
$\mathbf n(X) =\boldsymbol k, X \in D_1$. The following remarks are
key to  overcoming this difficulty.
\begin{remark}
 (i) Since for $\boldsymbol s =(s_1,s_2,s_3) \in \SS$, $|s_3|\neq 1$,
$$
\left|\frac{(s_3+1)s_i}{ s_1^2+s_2^2}\right| \leq
\frac{2|s_i|}{{ s_1^2+s_2^2}}  \leq \frac{
2}{\sqrt{s_1^2+s_2^2}}, \quad i=1,2,
$$
 $\omega_i \in L^p(\mathbb S^2)$  for $p<2$, if $|\nabla\nn| \in L^\infty(D_1)$.

(ii) Since  rotation of  the Cartesian  system in which $\mathbb S$ is
embedded changes the location of the poles of $\SS$,  M\"uller \&
S\u{v}er\'ak \cite{sverakmuller}  showed that singularities in
\eqref{leila10} can be dealt with   by integrating over a set of
rotated coordinates.

(iii)
 For $ (\cos\varphi\sin
    \vartheta, \sin\varphi\sin\vartheta, \cos\theta) = \boldsymbol s \in \SS$, and $\boldsymbol k = (0,0,1)$,
\begin{align}\int_{\boldsymbol s \in\SS}\frac{ dS_{\boldsymbol s}}{|\boldsymbol s - \boldsymbol k|} = \int_0^\pi \frac{2\pi\sin \theta}{2\sin(\theta/2)} ~d\theta=  2\pi\int_0^\pi \cos (\theta/2) \,d\theta = 4\pi.
\qquad \Box\label{pint}\end{align}
\end{remark}
Here the approach is similar to \cite{sverakmuller}, except that in
what follows the field $\mathbf n$ is rotated instead of the
coordinate system. To see that this is possible without changing
$\Phi$, let $\mathbf U$ be  a rotation matrix (a $3\times 3$ orthogonal matrix with determinant 1), the
transpose of the columns  of which form an orthonormal basis
$U_i,\, i= 1,2,3$, for $\mathbb R^3$ with
$$U_1= U_2\times U_3,\quad U_2= U_3\times U_1,\quad U_3= U_1\times U_2.
$$
Then, for a given $\mathbf n:D_1\to \SS$, let   $(\mathbf U \mathbf
n)(X) = \mathbf U (\mathbf n(X))\in \SS,\,X \in D_1$, and put
\begin{equation*}\label{leila05}
\mathbf m(X)=(\mathbf U \mathbf n)(X) = n_1(X) U_1+ n_2(X) U_2 +
n_3(X) U_3, \quad X\in D_1.
\end{equation*}
 It follows that
\begin{multline*}
\partial_1\mathbf m\times \partial _2\mathbf m =
(\partial_1 n_2\partial_2 n_3-\partial_2 n_2\partial_1n_3)U_1+
(\partial_1 n_3\partial_2 n_1-\partial_2 n_3\partial_1n_1)U_2\\+
(\partial_1 n_1\partial_2 n_2-\partial_2 n_1\partial_1n_2)U_3=
\mathbf U(\partial_1\mathbf n\times\partial_2\mathbf n),
\end{multline*}
and therefore, since $\mathbf U$ is orthogonal,
\begin{equation*}\label{leila25a}
\mathbf m\cdot(\partial_1 \mathbf m\times \partial_2\mathbf
m)=(\mathbf U \mathbf n)\cdot\big(\mathbf U(\partial_1 \mathbf
n\times
\partial_2\mathbf n)\big) = \mathbf n\cdot(\partial_1 \mathbf n\times
\partial_2\mathbf n)=\Phi.
\end{equation*}
Therefore, replacing   $\mathbf n$ with $\mathbf m$ in  formula
\eqref{leila10}
 yields
\begin{equation}\label{leila11}
\Phi = \mathbf n\cdot(\partial_1 \mathbf n\times \partial_2\mathbf
n)=
\partial_2 W_1 -\partial_1 W_2,
\end{equation}
where $W_i =  W_i(\nn,\mathbf U)$ is a function of $X$ given by
\begin{equation}\label{leila12}
    W_i(\nn,\mathbf U)=\frac{m_3+1}{m_1^2 +m_2^2}\Big(m_1\partial_i
    m_2-m_2\partial_i m_1\Big).
\end{equation}

Clearly, as with \eqref{leila10b},  the singularities  of $W_i$
occur at points of $D_1$ where $\mathbf m(X)= \pm\boldsymbol k$,
equivalently where $\mathbf n(X)= \mathbf U^{-1}\boldsymbol k$.
Therefore  \eqref{leila11} holds pointwise at $X \in D_1$ only  if
$\mathbf m(X)\neq \boldsymbol k$. In \eqref{leila12}, $(m_1,m_2,m_2)
= \mathbf m :D_1\to \SS$ depends on $\nn:D_1 \to \SS$ and on the
rotation matrix $\mathbf U$.

The next step is to parameterize a  suitably  family of rotation
matrices. So   for fixed $\bn' \in \mathbb S^2_0  = \SS \setminus
\{\boldsymbol k,-\boldsymbol k\}$,  let
\begin{equation}\label{leila14}
    \bn'=(n'_1,n'_2,n'_3),~ n'_1=\cos\varphi'\sin
    \vartheta',~
    n'_2=\sin\varphi'\sin\vartheta', ~  n'_3=\cos\vartheta',
\end{equation}
where $\cos\vartheta'\neq \pm1$, and
 let $\bf U(\bn')$ denote the rotation matrix the transpose (equivalently the inverse) of which is
\begin{align}\label{leila15}
\mathbf U(\bn')^\top&=
 \left(
\begin{array}{ccc}
\displaystyle { \cos\vartheta'\cos\varphi'} &
\displaystyle{-\sin\varphi'} & \displaystyle{\sin\vartheta'\cos\varphi'} \\
\displaystyle { \cos \vartheta'\sin\varphi'} &\displaystyle{
\cos\varphi'} &\displaystyle{\sin\vartheta'\sin\varphi'}
\\ \displaystyle {-\sin\vartheta'} &
\displaystyle{0}&\displaystyle{\cos\vartheta'
 }
  \end{array}
\right)
\\
&=\left(
\begin{array}{ccc}
\displaystyle { {\lambda'}^{-1/2}\, n'_1n'_3} &
\displaystyle{- {\lambda'}^{-1/2}\, n'_2} & \displaystyle{n'_1} \\
\displaystyle { {\lambda'}^{-1/2}\, n'_2n'_3} &\displaystyle{
{\lambda'}^{-1/2}\, n'_1} &\displaystyle{n'_2}
\\ \displaystyle {-{\lambda'}^{1/2}} &
\displaystyle{0}&\displaystyle{n'_3
 }
  \end{array}
\right),\notag
\end{align}
when  $\lambda':= {n'_1}^2+{n'_2}^2 = 1-{n_3'}^2\neq 0$. Thus  $\bn'
\mapsto \mathbf U (\bn')$ in \eqref{leila15} depends real analytically on
$(n'_1,n'_2,n'_3) ={\bn'}\in \mathbb S^2_0$.   Now, for  ${\bn'} \in
\mathbb S_0^2$ and ${\bn \in\SS}$, let
\begin{subequations}\label{leila16}
\begin{align} \mathbf U(\bn')\bn  &= {\boldsymbol  m} = (m_1,m_2,m_3) \in \SS
\\ \intertext{and define $\Gamma:\mathbb S^2\times \mathbb
S^2_0\times\mathbb R^3\to \mathbb R$  by}\notag
    \Gamma(\bn, \bn', \boldsymbol \xi)&=\frac{m_3+1}{m_1^2+m_2^2}
    \Big(m_1\big(\mathbf U(\bn')\boldsymbol \xi\big)_2-m_2\big(\mathbf U(\bn')\boldsymbol \xi\big)_1\Big) \\&=\frac{1}{1-m_3}
    \Big(m_1\big(\mathbf U(\bn')\boldsymbol \xi\big)_2-m_2\big(\mathbf U(\bn')\boldsymbol \xi\big)_1\Big),
\end{align}
\end{subequations}
where $\Gamma(\bn, \bn', \cdot):\RR^3 \to \RR$ is linear for fixed
$({\bn, \bn'}) \in \mathbb S^2\times \mathbb S^2_0$.

For fixed $\bn' \in \mathbb S^2_0$ and smooth $\mathbf n: D_1
\to \SS$, put $\mathbf m(X) = \mathbf U(\bn')(\nn(X))$  and note that $\partial_i
\mathbf m(X)) = \partial_i \big(\mathbf U(\bn')\nn\big)(X) = \mathbf
U(\bn')\big(\partial_i\nn(X)\big)$. Then, by \eqref{leila12} and
\eqref{leila16},
\begin{multline}\label{leila25}
 \omega_i(\nn, \bn')(X): = W_i(\nn, {\bf U}(\bs n'))(X) =\Gamma(\mathbf n(X), \bn', \partial_i
    \mathbf n(X))\\
=\frac{m_3(X)+1}{m_1^2(X)+m_2^2(X)}\Big(m_1(X)\partial_i
m_2(X)-m_2(X)\partial_i m_1(X)\Big)\\
=\frac{1}{1-m_3(X)}\Big(m_1(X)\partial_i
m_2(X)-m_2(X)\partial_i m_1(X)\Big),
 ~{i=1,2},
    \end{multline}
    when $\bs m \neq  \pm \bs k$.
The proof of Theorem \ref{aelita22} depends on estimates of $\Gamma$
and $\omega_i$ in terms on $\bn'\in \SS_0$ and $\nn:D_1 \to \SS$ in
the next section.

\subsection{Estimates of $\boldsymbol \Gamma$ and $\boldsymbol{\omega_i,\,i=1,2}$}\label{estimates}
\begin{definition}\label{leiladefinition}  Let
\begin{equation*}\label{leila18}
 \Sigma = \{(\bn, \bn') \in \mathbb S^2 \times \mathbb S^2_0: \bn \neq \bn'\}.
\qquad\qquad \qquad\qquad \Box\end{equation*}
\end{definition}

\begin{lemma}\label{leila19} The function $\Gamma$
in \eqref{leila16} is real-analytic on $\Sigma\times \mathbb R^3$
and
\begin{equation}\label{leila20}
|\Gamma(\bn, \bn', \boldsymbol \xi)|\leq \frac{2|\boldsymbol
\xi|}{|\bn-\bs n'|}\, \text{~for all~} (\bn,\bn')\in \Sigma,~
\boldsymbol \xi \in \mathbb R^3.
\end{equation}
\end{lemma}

\begin{proof}

 Since
$ \bn'\mapsto
   \mathbf U(\bn') $ and   $(\bn, \bn')\mapsto \mathbf U(\bn')\bn = \boldsymbol m = (m_1,m_2,m_3)\in \mathbb S^2$ are real analytic on $\SS\times \SS_0\times \mathbb R^3$,  $\Gamma$ in \eqref{leila16}  will be real analytic on $\Sigma \times \RR^3$ if
$$\displaystyle(\bn, \bn') \mapsto \frac{m_3+1}{m_1^2+m_2^2}= \frac{1}{1-m_3},\text{  since }m_1^2+m_2^2+m_3^2 =1,$$
is real-analytic on $\Sigma$.
So it suffices to observe that $m_3 \neq 1$ when $(\bn,\bn') \in\Sigma$.  This is because, by \eqref{leila14} and \eqref{leila15},
$\mathbf U(\bn')^{-1}(\boldsymbol k)=\bn'$ and hence, since $\bn
\neq \bn'$,
\begin{equation}\label{leila23}
  0<  |\bn'-\bn|= |\mathbf U(\bn')^{-1} (\boldsymbol k-\mathbf U(\bn')\bn)|=|\boldsymbol k-\mathbf U(\bn')\bn|
    =|\boldsymbol k
    -\boldsymbol m|.
\end{equation}
Now,
 when $(\bn, \bn') \in \Sigma$
and $\boldsymbol m \neq \boldsymbol \pm k$, it follows   from \eqref{leila16}(b),
\begin{align*}\big|\Gamma(\bn,\bn',\boldsymbol \xi)\big| &=
\left|
    \frac{m_3+1}{m_1^2+m_2^2}
    \Big(m_1\big(\mathbf U(\bn')\boldsymbol \xi\big)_2-m_2\big(\mathbf U(\bn')\boldsymbol \xi\big)_1\Big)\right|
\\
&\leq \left(
    \frac{m_3+1}{m_1^2+m_2^2}\right)
    \sqrt{m_1^2+m_2^2}\,\big|\mathbf U(\bn') \xi\big| =
    \frac{m_3+1}{\sqrt{m_1^2+m_2^2}}
    \,|\boldsymbol \xi|  \\&=
    \frac{\cos \theta +1}{\sin \theta}\,|\boldsymbol \xi| = \frac{\cos (\theta/2)}{\sin (\theta/2)}|\boldsymbol \xi| \leq \frac{|\boldsymbol \xi|}{\sin (\theta/2)}\\& =  \frac{2|\boldsymbol \xi|}{|\boldsymbol k
    -\boldsymbol m|} = \frac{2|\boldsymbol \xi|}{|\bn'-\bn|},
\end{align*}
where, in spherical polar coordinates, $m_1 =\cos\phi\sin
    \theta,\,
    m_2=\sin\phi\sin\theta,\, m_3=\cos\theta$, $\theta \in (0,\pi)$. This shows \eqref{leila20} and  completes the proof.
    \end{proof}

\begin{definition}\label{gauss}
For  $\bn' \in \SS_0$ and smooth $\nn:D_1 \to \SS$, let
$$Z(\nn,\bn') = \{X \in D_1: \nn(X) \neq \bn'\}= \{X\in D_1: (\nn(X),\bn')
\in \Sigma\}.\qquad \Box$$
\end{definition}

\begin{lemma}\label{leila26} For fixed $\bn' \in \SS_0$ and smooth $\nn:D_1 \to \SS$, the functions $\omega_i(\nn, \bn'),\,i=1,2$, in \eqref{leila25}  are infinitely differentiable
at $X \in Z(\nn,\bn')$ where
\begin{equation}\label{leila27}
\Phi=\partial_2\omega_1(\nn, \bn')-\partial
    _1\omega_2(\nn, \bn') \text{ and }
    |\omega_i(\nn, \bn')(X)|\leq \frac{2|\partial_i \mathbf n(X)|}{|\mathbf n(X)-\bn'|}.
    \end{equation}
\end{lemma}
\begin{proof}
For fixed $\bn'\in \SS_0$, the mapping  $X\to \mathbf n(X)$ takes
$Z(\nn,\bn')$ to $\mathbb S^2\setminus \{\bn'\}$ and since $
(\mathbb S^2\setminus \{\bn'\})\times \{\bn'\}\subset \Sigma, $ it
follows from Lemma \ref{leila19} that $\Gamma(\bn, \bn', \boldsymbol
\xi)$ is real analytic with respect to $(\bn,\boldsymbol\xi)$ at
$(\nn(X), \boldsymbol \xi_0)$, when $X \in Z(\nn, \bn')$ and
$\boldsymbol \xi_0\in \RR^3$ is arbitrary. Therefore, since $\nn$ is
infinitely differentiable on $D_1$, it follows from \eqref{leila25}
that  the functions $\omega_i(\cdot, \bn')$, $i=1,2$,  are infinitely
differentiable on $Z(\nn,\bn')$.  Since,
by \eqref{leila12},
$
\omega_i(\nn, \bn')(X) = W_i(\nn, \mathbf U(\bn'))(X),
$
 where $ \mathbf
m (X)= \mathbf U(\bn')\mathbf n(X)$ in the definition of $W_i$, and  $\mathbf m(X)=\pm \boldsymbol k$ if and only if $\mathbf
n(X)= \pm \bn'$m, because
$\mathbf U^{-1}(\bn')\boldsymbol k=\bn'$ by  \eqref{leila15}, it follows that
\begin{align}\notag
\Phi(X) &= \partial_2W_1(\nn,{\bf U}(\bn '))(X) - \partial_1 W_2(\nn,{\bf U}(\bn '))(X) \text{ when }\mathbf m(X)\neq \pm\boldsymbol k,\\
&= \partial_2\omega_1(X,\bs) - \partial_1 \omega_2(X,\bn), \quad\text{ when }\nn(X) \neq \pm \bn'.\label{leila30}
\end{align}
Thus the equality in \eqref{leila27} follows from \eqref{leila11} and
\eqref{leila12} when $\nn(X) \neq \pm \bn'$
 and, since $X \in Z(\nn,\bs n')$, it remains only   to consider the case when $\nn(X) = -\bs n'$.

Let $L$ denote the level set $\{X \in D_1: \nn(X) = -\bs n'\}$, at every point of which $\omega_i(\cdot ,
\mathbf n')$, $i=1,2,$ is infinitely differentiable because $L$ is a compact subset of $Z(\nn, \bn')$. Moreover,
$\omega_i(\cdot, \mathbf n')$ is zero almost everywhere on $L$ because $\partial_i \mathbf n$  is zero almost everywhere on  $L$. Now recall that if  a function is
infinitely differentiable in a neighborhood of  a set of positive
measure and is zero on that set, then the derivative of the
function is zero almost everywhere  on the set. It follows that
$\partial_j\omega_i(X,\mathbf n')=0$, $i,j=1,2$, almost everywhere on $L$, and the same conclusion can be drawn for
the function $\Phi$.

Hence the equality  \eqref{leila30} holds almost everywhere
on $L$. When combined with the original version of
\eqref{leila30} it follows that \eqref{leila30} holds  almost everywhere on $Z(\nn, \bn')$. Since both
sides of this equality are smooth on this set,
\eqref{leila30} holds everywhere on $Z(\nn, \bn')$.
 The inequality  \eqref{leila27} now follows from \eqref{leila20} and the proof is complete.
\end{proof}


Lemma \ref{leila26} concerns  smoothness of  $\omega_i$ at   $X \in
Z(\nn, \bs n')$  for fixed $\bn'$. The next lemma deals with  their joint smoothness
with respect to  $X$ and $\bn'$.

\begin{lemma}\label{leila60}
Suppose that $\nn:D_1 \to \SS$ is  smooth, that $
K\subset \mathbb S^2_0$ and   $  E\subset \Sigma$ are compact, and that $G \subset D_1$ is such  that
\begin{equation*}\label{leila61}
    (\mathbf n(X), \bn')\in   E\text{ for all } X\in
      G \text{ and } \bn'\in   K.
\end{equation*}
Then there is a neighborhood $  O_  K$ of $  K$ such that the
functions $\omega_i$, $i=1,2,$ are infinitely
differentiable with respect to  $(X,\bn')$ on $  G\times O_K$.
\end{lemma}
\begin{proof}
Since $\mathbf n$ is smooth on $D_1$,  which is compact,
$|\partial_i\mathbf n|\leq b$ on $  D_1$  for some $b<\infty$ . Let $
B=\{\boldsymbol \xi: |\boldsymbol \xi| \leq b,\,\boldsymbol \xi
\in\mathbb R^3\}$ and  $  V=  E\times
  B$. Since $  V$ is compact in $ \Sigma\times \mathbb R^3$, by
Lemma \ref{leila19}, the function $\Gamma$ is real-analytic on a
neighborhood $  O_  V$ of $  V$. Moreover, the mapping
$$
(X, \bn')\to (\mathbf n(X), \bn', \partial_i\mathbf n(X))
$$
is infinitely differentiable on $D_1\times \SS_0$ and maps $ G\times
K$ into $  V$. The  smoothness of  $\omega_i$ on $  G\times O_{  K}$
follows from
 \eqref{leila25}.
\end{proof}
\subsection{Proof of Theorem \ref{aelita22}} \label{asol3}
As noted at the beginning of Section \ref{asol}, it suffices to
consider smooth
 maps $X\mapsto \mathbf n(X)$ which takes
the disc $D_1$ into the unit sphere  $\mathbb S^2$  and, by
hypothesis  \eqref{aelita23}, satisfy $\text{meas}\,  \big( \mathbf
n(D_1)\big)<4\pi-\delta$ (see \eqref{aelita21}).

Let $  A= \nn(D_1)\cup\{\pm \boldsymbol k\}\subset\mathbb S^2$. Then
$\text{meas}\,(\mathbb S^2\setminus   A)>\delta$, since
$\text{meas}\,(\mathbb S^2)= 4\pi$ and $\text{meas}\,(
A)<4\pi-\delta$.
 Hence there is a compact $K \subset \mathbb S^2\setminus   A$
with $\text{meas}\,( K)\geq\delta$.
 Since $D_1$ is compact and $\nn$ is continuous on $D_1$,  $\nn(D_1)$ is compact
and, for some  $\sigma>0$ independent of $\bn'\in K$ and  of $X\in
D_1$,
\begin{equation*}\label{leila13}
|\bn'\pm\boldsymbol k|>\sigma\text{ and }   |\mathbf n(X)-\bn'|>
\sigma\text{ for all } X\in D_1
    \text{ and  }\bn'\in   K,
\end{equation*}
since (see Definition \ref{gauss}), $Z(\nn, \bn')= D_1$ for all
$\bn'\in   K$.  From  Lemma \ref{leila60} with $  G=D_1$ and $E =
D_1\times K$, it follows that there is a neighborhood $  O_ K$
of  $  K$ such that $\omega_i(\nn, \bn'),\,i=1,2$, are infinitely
differentiable on $D_1\times   O_  K$. Therefore by
Lemma \ref{leila26}
\begin{align}\label{leila65}
\Phi(X)=\partial_2\omega_1 (\nn, \bn')(X)-\partial_1\mathbf
\omega_2(\nn, \bn')(X),~ X\in D_1,~ \bn'\in   K,
\\
\intertext{and}\notag |\omega_i(\nn, \bn')(X)|\leq \frac{2}{|\mathbf
n(X)-\bn'|}|\partial_i \mathbf n(X)|,~ X\in D_1,~\bn'\in   K,
\end{align}
where $|\mathbf n(X)-\bn'|\geq \sigma>0$, $(X, \bn') \in D_1\times
  K$. Now for $X \in D_1$ let
\begin{equation*}\label{leila66}
    \Omega_i(X)=\frac{1}{\text{meas}\,(  K)}\int_{
    K}\omega_i(\nn, \bn')\, dS_{\bn'},
\end{equation*}
where the integration over $K$ is with respect to the measure on
$\SS$. Then  \eqref{leila27} yields the estimate
\begin{align*}
    |\Omega_i(X)|&\leq \frac{2}{\text{meas}\,(  K)}\left(\int_{
    K}\frac{dS_{\bn'}}{|\mathbf n(X)-\bn'|}\right)\,|\partial_i \mathbf n(X)|\\
&\leq \frac{2}{\text{meas}\,(  K)}\left(\int_{
    \SS}\frac{dS_{\bn'}}{|\mathbf n(X)-\bn'|}\right)\,|\partial_i \mathbf n(X)|\\& =
\frac{8\pi}{\text{meas}\,(  K)}\,|\partial_i \mathbf n(X)| \leq
\frac{8\pi}{\delta}\,\,|\partial_i \mathbf n(X)|,
\end{align*}
which gives
\begin{equation*}
    \|\Omega_i\|_{L^2(D_1)}\leq  \frac{8\pi}{\delta}\,\,\|\partial_i \mathbf
    n\|_{L^2(D_1)}.
\end{equation*}
Multiplying   \eqref{leila65} by $\zeta \in C_0(D_1)$ and
integrating with respect to $\bn'$ over   $K \subset \SS$ yields
\eqref{aelita24}. This completes the proof. \qed

\subsection{Proof of H\'elein's Conjecture, $\bn \bf = \bf 3$}\label{inna}

In this section    {an improved version of H\'elein's Conjecture for $n=3$ (see page \pageref{inna2}), with hypothesis \eqref{inna1} replaced by \eqref{inna1+}, is deduced from Theorem \ref{aelita22}} using a
continuation argument similar to that  in \cite{helein}.

\textbf{Step 1. A Priori Bounds.}\\
The first observation is similar to that of Section \ref{isothermal}.
\begin{lemma}
Suppose, for  smooth $\mathbf n: D_1\to \mathbb S^2$, there exist an
orthonormal frame $(\mathbf e_1, \mathbf e_2)$ were the $\mathbf
e_i$ are smooth on $D_1$ and orthogonal to $\mathbf n$, and
$(\mathbf e_1, \mathbf e_2, \mathbf n)$ has positive orientation.
Suppose also that  $f$ is smooth and satisfy \eqref{inna4} and
\eqref{inna5}. Then
\begin{equation}\label{aelita16}
-\Delta f =\mathbf n\cdot (\partial_1\mathbf n\times
\partial_2\mathbf n)\text{~in~} D_1, \quad
f=0\text{~on~}\partial D_1,
\end{equation}and
\begin{equation}\label{aelita17}\begin{split}
    \partial_1\mathbf e_1=-\partial_2 f\, \mathbf e_2-(\mathbf
    e_1\cdot \partial_1\mathbf n)\,\mathbf n,\qquad
\partial_2\mathbf e_1=\partial_1 f\, \mathbf e_2-(\mathbf
    e_1\cdot \partial_2\mathbf n)\,\mathbf n,\\
\partial_1\mathbf e_2=\partial_2 f\, \mathbf e_1-(\mathbf
    e_2\cdot \partial_1\mathbf n)\,\mathbf n,\qquad
\partial_2\mathbf e_2=-\partial_1 f\, \mathbf e_1-(\mathbf
    e_2\cdot \partial_2\mathbf n)\,\mathbf n.
\end{split}\end{equation}

\end{lemma}

\begin{proof}
Since $\mathbf n\cdot\partial_i\mathbf n=\partial_i(\mathbf n\cdot
\mathbf e_j)=0$ for all $i,j$, it follows that
$$\partial_i \mathbf n=-( \mathbf n\cdot \partial_i\mathbf e_i)\mathbf e_i -(\mathbf n\cdot \partial_i\mathbf e_j)\mathbf e_j,~i \neq j,$$
and hence that
\begin{multline}\label{inna9}
 \partial_1\mathbf n \times \partial_2\mathbf n=\big[(\mathbf n\cdot \partial_1\mathbf e_1)\mathbf e_1+(\mathbf n\cdot \partial_1\mathbf e_2)\mathbf e_2\big]
 \times \big[(\mathbf n\cdot \partial_2\mathbf e_1)\mathbf e_1+(\mathbf n\cdot \partial_2\mathbf e_2)\mathbf e_2\big] \\
= \big((\mathbf n\cdot
\partial_1\mathbf e_1)(\mathbf n\cdot \partial_2\mathbf e_2)-
 (\mathbf n\cdot \partial_1\mathbf e_2)(\mathbf n\cdot \partial_2\mathbf e_1)\big) \mathbf n\,.
\end{multline}
Moreover,
$$
\partial_1\mathbf e_1\cdot \partial_2\mathbf e_2=(\mathbf n\cdot \partial_1\mathbf e_1)
(\mathbf n\cdot \partial_2\mathbf e_2),
$$
since
$$
\partial_1 \mathbf e_1=(\partial_1 \mathbf e_1\cdot\mathbf e_2)\mathbf e_2+
(\partial_1 \mathbf e_1\cdot\mathbf n)\mathbf n,\,\,\,
\partial_2 \mathbf e_2=(\partial_2 \mathbf e_2\cdot\mathbf e_1)\mathbf e_1+
(\partial_2 \mathbf e_2\cdot\mathbf n)\mathbf n,
$$
and similarly
$$
\partial_2\mathbf e_1\cdot \partial_1\mathbf e_2=(\mathbf n\cdot \partial_2\mathbf e_1)
(\mathbf n\cdot \partial_1\mathbf e_2).
$$
Substituting these observations into \eqref{inna9} gives
\begin{equation*}\label{inna10}
(\partial_1\mathbf n \times \partial_2\mathbf n)\cdot \mathbf
n=\partial_1\mathbf e_1\cdot \partial_2\mathbf e_2-\partial_1\mathbf
e_2\cdot \partial_2\mathbf e_1,
\end{equation*}
and \eqref{aelita16} follows from \eqref{inna5}. Next,
\begin{align*}
\partial_j\mathbf e_1&=(\partial_j \mathbf e_1\cdot \mathbf
e_2)\mathbf e_2+(\partial_j \mathbf e_1\cdot \mathbf n) \mathbf n,
&\partial_j\mathbf e_1\cdot\mathbf e_2&=-\mathbf
e_1\cdot\partial_j\mathbf e_2,
\\
\partial_j\mathbf e_2&=(\partial_j \mathbf e_2\cdot \mathbf
e_1)\mathbf e_1+(\partial_j \mathbf e_2\cdot \mathbf n) \mathbf n,
 &\partial_j\mathbf e_i\cdot\mathbf n~&=-\mathbf
e_i\cdot\partial_j\mathbf n,
\end{align*}
and \eqref{inna4} imply
  \eqref{aelita17}.\end{proof}

\begin{lemma}\label{inna12} For a smooth $\nn:D_1 \to \SS$  {which satisfies \eqref{aelita23}}, let  $(\mathbf e_1, \mathbf e_2, \mathbf n)$ be an orthonormal moving frame with
positive orientation which, together with a smooth function $f:D_1
\to \RR$,  satisfies  \eqref{inna4} and \eqref{inna5}.
 Then
\begin{gather}\label{inna13}
\|\mathbf e_{i}\|_{C^k(D_1)}+\|f\|_{C^k(D_1)}\leq C(k,\nn), \quad
k\geq 2,\\\label{inna14} \|\nabla\mathbf e_{i}\|_{L^2(D_1)}+\|\nabla
f\|_{L^2(D_1)}+ \|f\|_{L^\infty(D_1)}\leq c(\delta).
\end{gather}
Here $C(k,\nn)$ depends only on $k$ and $\mathbf n$, and
 $c(\delta)$ only on   {$\delta$ in
\eqref{aelita23}}.
\end{lemma}
\begin{proof} By \eqref{aelita16} of  the preceding Lemma, 
\begin{equation}\label{inna15}
    -\Delta f=(\partial_1\mathbf n \times \partial_2\mathbf n)\cdot \mathbf
n=: \Phi\text{~in~} D_1,\quad f = 0 \text{ on } \partial D_1,
\end{equation}
and it follows from standard estimates of solutions of the
Dirichlet problem for Poisson's equation that
\begin{equation}\label{inna16}
    \|f\|_{C^k(D_1)}\leq c(k, \|\mathbf n\|_{C^k(D_1)}).
\end{equation}
In particular,
\begin{equation*}\label{inna17}
    \|\nabla f\|_{C^{k-1}(D_1)}\leq c(k, \|\mathbf n\|_{C^k(D_1)}),
\end{equation*}
from which it follows by \eqref{aelita17} and induction that
\begin{equation*}
    \|\nabla \mathbf e_i\|_{C^{k-1}(D_1)}\leq c(k, \|\mathbf
    n\|_{C^k(D_1)}),
\end{equation*}
which  with \eqref{inna16} implies \eqref{inna13}.  {By \eqref{aelita23}}, $\nn$ satisfies the hypotheses of Theorem
\ref{aelita22}, and hence by \eqref{aelita25},
$$
\|\Phi\|_{W^{-1,2}_0(D_1)}\leq \frac{c}{\delta}\|\nabla \mathbf
n\|_{L^2(D_1)}\leq \frac{\sqrt{8\pi}  c}{\delta}.
$$
This and \eqref{inna15}  yields that $
    \|\nabla f\|_{L^2(D_1)}\leq c(\delta)
$ and, when combined with \eqref{aelita17},
\begin{equation*}
    \|\nabla \mathbf e_j\|_{L^2(D_1)}\leq c(\delta).
\end{equation*}
It now follows  \eqref{inna15} and the Went-Topping inequality,
\cite[Thm. 1]{topping} and \cite{wente}, that $
    \|f\|_{L^\infty(D_1)}\leq c(\delta)
$. Hence \eqref{inna14} holds and the proof is complete.
\end{proof}

\textbf{Step 2. A Parameterized Family of Normal Vector Fields.}\\
For any smooth $\mathbf n:D_1 \to \SS$  {satisfying \eqref{aelita23}},
consider the family of vector fields
\begin{equation*}\label{inna7}
    \mathbf n_\lambda(X) =\mathbf n(\lambda X), \quad \lambda\in
    [0,1],\quad X \in D_{1}.
\end{equation*}
Note that  $\mathbf n_0$ is a constant vector field and that, for
all  $\lambda \in [0,1]$, $\mathbf n_\lambda$ satisfies
\begin{equation*}\label{inna8}
\begin{split}   \| \mathbf n_\lambda\|_{C^k(D_1)}&\leq \|\mathbf n\|_{C^k(D_1)},
   \, \, k\geq 0,\\ \int_{D_{1}}|\pd_1 \nn_\lambda\times\pd_2
   \nn_\lambda|dX&= \int_{D_{\lambda}}|\pd_1 \nn\times\pd_2
   \nn|\,dX
  \leq
   4\pi-\delta.\end{split}
\end{equation*}

\begin{corollary}\label{inna12a} Let $f_\lambda, \mathbf e_{\lambda,i}\in C^\infty(D_1)$
be a solutions to equations \eqref{inna4}-\eqref{inna5} with
$\mathbf n$ replaced by $\mathbf n_\lambda$. Then
\begin{gather*}
\|\mathbf e_{\lambda, i}\|_{C^k(D_1)}+\|f_{\lambda}\|_{C^k(D_1)}\leq
C(k,\nn), \quad k\geq 2,\\ \|\nabla\mathbf
e_{\lambda,i}\|_{L^2(D_1)}+\|\nabla f_\lambda\|_{L^2(D_1)}+
\|f_\lambda\|_{L^\infty(D_1)}\leq c(\delta),
\end{gather*}
where the constant $C(k,\nn)$ depends on $k$ and $\mathbf n$, but is
independent of $\lambda$,  and $c(\delta)$ depends only on $\delta$
 {in
 \eqref{aelita23}}.
\end{corollary}
\begin{proof} The proof is the same as for  $\lambda = 1$ in Lemma \ref{inna12}.
\end{proof}
\textbf{Step 3. Parameter Continuation.} \\ Denote by $\mathcal L$
the set of $\lambda\in [0,1]$ for which the system
\eqref{inna4},\eqref{inna5} has an infinitely differentiable
solution $\{f_\lambda, \mathbf e_{\lambda, i}:\,i=1,2\}$, and note
that $0\in \mathcal L$. Indeed, since $\mathbf n_0=\text{const.}$,
the function $f_0=0$ and an arbitrary pair of constant vectors
$\mathbf e_{0,i}$ with
$$
\mathbf e_{0,i}\cdot \mathbf e_{0,j}=\delta_{ij}, \quad \mathbf
e_{0,i}\cdot\mathbf n_0=0,\quad i=1,2,
$$
satisfy  \eqref{inna4} and \eqref{inna5} with $\mathbf n = \mathbf
n_0$.

To show that $\mathcal L$ is closed let $\lambda_n \in\mathcal L$
and $\lambda_n \to \lambda$ as $n\to \infty$. Then by  Corollary
\ref{inna12a} there is a sequence $\{n_\ell\} \subset \mathbb N$
such that the solutions, $\mathbf e_{\lambda_{n_{\ell},i}},\,i=1,2$,
and $f_{\lambda_{n_\ell}}$ to problem \eqref{inna4}-\eqref{inna5}
with $\mathbf n$ replaced by $\mathbf n_{\lambda_{n_\ell}}$,
converge,  in  $C^k(D_1)$ for all $k$, to functions denoted by
$\mathbf e_{\lambda, i}$ and $f_\lambda$. Obviously these functions
are infinitely differentiable and satisfy equations
\eqref{inna4}-\eqref{inna5} with $\mathbf n$ replaced by $\mathbf
n_\lambda$. Hence $\mathcal L$ is closed in $[0,1]$.

Now, following  \cite{helein}, to show $\mathcal L$ is open let
$\lambda_0\in \mathcal L$ and $I_0\subset [0,1]$ be a segment with
endpoints $\lambda_0$ and $\lambda_0+t_0$, $0<|t_0|<1$. The goal is
to prove that, for sufficiently small $t_0$ and all $t\in   I_0$,
equations \eqref{inna4} and \eqref{inna5} with $\mathbf n=\mathbf
n_{\lambda_0+t}$ have a smooth solution. To simplify  notation,  let
$\mathbf n^0$ and $\mathbf n^t$ denote $\mathbf n_{\lambda_0}$ and
$\mathbf n_{\lambda_0+t}$, and denote solutions of \eqref{inna4} and
\eqref{inna5} with $\mathbf n=\mathbf n^t$ by $\mathbf e^t_i$ and
$f^t$ .

Then, for $t \in I_0$ and $X\in D_1$, define a family of orthogonal
projections $\mathbb P^t(X):\mathbb R^2 \to \{\mathbf n^t(X)\}^\perp
\subset \mathbb R^2 $
 by
$$
\mathbb P^t(X)\boldsymbol \xi= \boldsymbol \xi-(\mathbf n^t(X)\cdot
\boldsymbol \xi)\, \mathbf n^t(X)\text{~for all~} \boldsymbol \xi\in
\mathbb R^2.
$$
Since $\lambda_0\in \mathcal L$, there exist $C^\infty$    vector
fields  $\mathbf e_i^0$, $i=1,2$, which satisfy  \eqref{inna4} and
\eqref{inna5} with $\mathbf n=\mathbf n^0$ and, since $\mathbf
n^0\in C^\infty(\mathbb R^2)$, there exists $t_0>0$ such that
$$
\|\mathbb P^t(X)-\mathbb P^0(X)\|\leq 1/8\text{~for~} X\in D_1.
$$
Then, for $t \in I_0$, let $ \overline{\mathbf e}^{(t)}_i(X)=
\mathbb P^t(X)\,\mathbf e_i^0(X), ~ i=1,2$,  and note, since
$\mathbb P^0\mathbf e_i^0=\mathbf e_i^0$, that for $X\in D_1$, and
$|t|\leq |t_0|$,
$$
3/4\leq |\overline{\mathbf e}^{(t)}_1|\leq 1, \quad
|\overline{\mathbf e}^{(t)}_2\cdot \overline{\mathbf e}^{(t)}_1|\leq
1/4.
$$
Therefore, with their dependence on $t \in I_0$ suppressed for
convenience of notation,  vector fields ${\bf e}^*_i$, $i=1,2$, are
well defined on $D_1$ by
$$\mathbf e_1^*= \frac{1}{|\overline{\mathbf e}^{(t)}_1|}\,\overline{\mathbf
e}^{(t)}_1, \quad \mathbf e_2^*=\frac{1}{|\overline{\mathbf
e}^{(t)}_2-(\overline{\mathbf e}^{(t)}_2\cdot \overline{\mathbf
e}^{(t)}_1)\overline{\mathbf e}^{(t)}_1|}\,\big(\overline{\mathbf
e}_2-(\overline{\mathbf e}^{(t)}_2\cdot \overline{\mathbf
e}_1)\overline{\mathbf e}^{(t)}_1\big).
$$
Since $\mathbf n^t\in C^\infty(D_1)$, it follows that $\mathbf
e_i^*\in C^\infty(D_1)$, and obviously the orthonormal triplets
$(\mathbf e_1^*, \mathbf e_2^*, \mathbf n^t)$ have positive
orientation, $t \in I_0$. Now, for any smooth function $\theta:D_1
\to \RR$ (which is to be determined),  define $\mathbf e^t_i$, $i=1,2$, by 
\begin{equation}\label{inna*}
\mathbf e^t_1+i\mathbf e^t_2=e^{i\theta}(\mathbf e_1^*+i\mathbf
e_2^*).
\end{equation}
Let $\mathbf e^*_i = (e^*_{i1},e^*_{i2})$. Then, since $\mathbf e^*_i
\cdot \mathbf e^*_j = \delta_{ij}$, $i,j=1,2$ and by\eqref{inna*},
\begin{equation}\label{inna**}\begin{split}
\mathbf  e_1^t &= (e^*_{11}\cos \theta -e^*_{21} \sin \theta , ~ e^*_{12} \cos \theta -e^*_{22}\sin \theta ),\\
\mathbf  e_2^t &= (e^*_{11}\sin \theta +e^*_{21} \cos \theta , ~
e^*_{12} \sin \theta +e^*_{22}\cos \theta ),
\end{split}
\end{equation}
 $(\mathbf e_1^t, \mathbf e_2^t,
\mathbf n^t)$ is an orthonormal triple with positive orientation.
The aim now is, for $t\in I_0$, to find a solution $\mathbf e^t_i$
of \eqref{inna4} and \eqref{inna5} in this form.

To do so, define vector fields on $D_1$ by
\begin{equation}\label{t*t}
\mathbf e_1^* \,d\mathbf e_2^*:=(\mathbf e_1^*\cdot \partial_1
\mathbf e_2^*,\, \,\mathbf e_1^*\cdot \partial_2 \mathbf
e_2^*),\quad \mathbf e_1^t \,d\mathbf e_2^t:=(\mathbf e_1^t\cdot
\partial_1 \mathbf e_2^t,\, \,\mathbf e_1^t\cdot \partial_2 \mathbf
e_2^t),
\end{equation}
and note from \eqref{inna**}, since $|\mathbf e_1^*(X)| = |\mathbf
e_2^*(X)| = 1$, $\mathbf e_1^*(X)\cdot\mathbf e_2^*(X) =0$ and  the
$\mathbf e_i^*$s are infinitely differentiable  on $D_1$, that
\begin{equation}\label{qc}
\mathbf e_1^t \,d\mathbf e_2^t=\nabla\theta +\mathbf e_1^*
\,d\mathbf e_2^*.
\end{equation}

 Next, as in
\cite{helein} 
 note that the variational problem
\begin{equation*}\label{inna20}
\inf\limits_{ \theta\in \mathcal V} \int_{D_1}\big|\nabla \theta
+\mathbf e_1^* \,d\mathbf e_2^*\big|^2\, dX ,\quad \mathcal V =
\Big\{ \theta\in W^{1,2}(D_1): ~ \int_{D_1} \theta\, dX=0\Big\},
\end{equation*}
 has a unique,
infinitely differentiable minimiser $\vartheta \in \mathcal V$ which
satisfies
$$\div \big(\nabla\vartheta +\mathbf e_1^*
\,d\mathbf e_2^*\big) = 0 \text{ in } D_1,\quad {\bs\nu}\cdot
(\nabla\vartheta +\mathbf e_1^* \,d\mathbf e_2^*) = 0 \text{ on }
\partial D_1,
$$
where $\boldsymbol\nu$ is the unit normal to $\partial D_1$. By
\eqref{qc} this can be re-written
\begin{equation}\label{inna21}
    \text{~div~}\mathbf e_1^t \,d\mathbf e_2^t=0\text{~in~} D_1,
    \quad \mathbf e_1^t \,d\mathbf
    e_2^t\cdot\boldsymbol \nu=0\text{~on~}\partial {D_1}.
\end{equation}
For  $t \in I_0$, let $\mathbf e_1^t \,d\mathbf e_2^t (X)=
(h^t_1(X), h^t_2(X)),~X \in D_1$, and put
$$f^t(X) = c  - \int_0^1 X\cdot \big(h^t_2(sX), -h_1^t(sX)\big) ds, \text{ where $c$ is a constant.}
$$
Then  $f^t\in C^\infty(D_1)$ and, since $\div
(h^t_1,h^t_2)=0$ by   \eqref{inna21},
$$
(h^t_1,h^t_2) = (\partial_2f^t, -\partial_1 f^t), \text{ or,
equivalently, } \mathbf e_1^t \,d\mathbf e_2^t(X)= \nabla^\perp f^t,
\quad X \in D_1.
$$
Hence, by \eqref{t*t},  $f^t$ and  $\mathbf e_i^t,\, i=1,2$, satisfy
\eqref{inna4}. 
Also, by the second part of  \eqref{inna21}, $f^t$ is constant on
$\partial D_1$, and the constant $c$ can be chosen so that $f^t=0$ on
$\partial D_1$. Therefore  \eqref{inna5} follows because
$$
\partial_1(\mathbf e^t_1\cdot \partial_2\mathbf e^t_2)-
\partial_2(\mathbf e^t_1\cdot \partial_1\mathbf e^t_2)
=\partial_1\mathbf e^t_1\cdot \partial_2\mathbf e^t_2-
\partial_2\mathbf e^t_1\cdot \partial_1\mathbf e^t_2.
$$
Since  $(\mathbf e_1^t, \mathbf e_2^t)$ and $f^t$ satisfy equations
\eqref{inna4} and \eqref{inna5} with $\mathbf n={\mathbf n}^t$ for all
$|t|\leq |t_0|$,  $\mathcal L$ is open in $[0,1]$. Since $\mathcal
L\neq \emptyset$ is also closed and $[0,1]$ is connected,
$\mathcal L =[0,1]$.
When $t=1$ this shows  $\mathbf e_1^1, \mathbf e_2^1, f^1$ satisfy
\eqref{inna4}  \eqref{inna5}, and \eqref{inna6} holds by Lemma
\ref{inna12}. Thus  {H\'elein's  Conjecture for $n=3$ with hypothesis \eqref{aelita23}} is established, and the proof
is complete.\qed

\section{ Beyond $\boldsymbol {8\pi}$}\label{krit000}

By developing the work of previous sections and using  some classical integral
 geometry, the aim  is to prove Theorem \ref{aelita30}    for  $\nn \in W^{1,2}(D_1,\SS)$.
\subsection{Preliminaries on Geometric Integration}

Let $\text{\rm card}(E)$ be the number of points in a finite set $
E$ and $\rm{card}(E)=\infty$ otherwise. Then for $E \subset\mathbb
R^m$, $\rm{card}(E)$ is finite when  its 0-Hausdorff measure
$\mathcal H^0(E)$  is finite, and if  $E$ consists of a finite
number of points $a_i\in \mathbb R^m$, every function $g$ defined on
$E$ is $\mathcal H^0$-measurable.

For fixed $\nn \in C^\infty(D_1,\SS)$ and  $\bn' \in \SS$, let
$$
Y(\nn,\bn') = \{X\in D_1: \nn(X) = \bn'\}.
$$
\begin{theorem}\label{krit2}
For $\nn \in C^\infty(D_1,\SS)$ and    $g\in L^1(D_1)$,
\begin{equation}\label{krit3}
    \int_{D_1} g(X) |\Phi(X)|\, dX=
    \int_{\mathbb S^2}\Big\{\sum_{\{A \in Y(\nn,\bn')\}}g(A)\Big\}\, d S_{\bn'},
\end{equation}
where $\Phi=\mathbf n\cdot (\partial_1\mathbf n
\times\partial_2\mathbf n)$.
\end{theorem}
\begin{proof} In Federer's \cite[Thm. 3.2.22]{federer} with $\nu = 3$ and $n=m=\mu=2,$ let
$$
 \quad W= D^1, \quad Z=\mathbb S^2, \quad
w=X, \quad z=\bn' \text{ and } f= \nn \in C^\infty(D_1,\SS).
$$
Then  $\text{\rm ap}\,J_\mu Df(w)=|\Phi|$,  $2$-dimensional Hausdorff
measure coincide with Lebesgue
measure on $D_1$ and  $\SS$,  and with  $\mathcal H^0$ denoting 0-dimensional Hausdorff measure, 
$$ \int_E g\,
\,d\mathcal H^0= \sum_{A \in E} g(A) \text{ when } \mathcal H^0(E) <
\infty. $$ 
Thus \eqref{krit3} is the statement of
\cite[Thm. 3.2.22 (3)]{federer} in this context.
\end{proof}

\begin{corollary}\label{krit4}
Under the hypotheses of  Theorem \ref{krit2},  for all Borel sets
$\mathcal A\subset \mathbb S^2$,
\begin{equation*}\label{krit5}
    \int_{\mathcal F} g(X) |\Phi(X)|\, dX=
    \int_{\mathcal A}\Big\{\sum_{A\in Y(\nn,\bn')} g(A)\Big\}\, d S_{\bn'}, \quad
    \mathcal F=\nn^{-1}(\mathcal A),
\end{equation*}
\end{corollary}

\begin{proof} If $\chi_{_{\mathcal F}}$ and $\chi_{_{\mathcal A}}$ are the  characteristic functions of
 $\mathcal F$ and $\mathcal A$, $\chi_{_{\mathcal F}}(X)=
\chi_{_{\mathcal A}}(\mathbf n(X))$ and, in particular,
$\chi_{_{\mathcal F}}(X)=\chi_{_{\mathcal A}}(\bn')$ for all $ X\in
Y(\nn,\bn')$. It follows from \eqref{krit3} with $g$ replaced by
$\chi_{_{\mathcal F}}\, g$ that
\begin{multline*}
\int_{\mathcal F} g(X)|\Phi(X)|\, dX=\int_{D_1} g(X)\chi_{_{\mathcal
F}} (X)|\Phi(X)|\, dX \\= \int_{\mathbb S^2}\Big\{\sum_{A\in
    Y(\nn,\bn')}\chi_{_{\mathcal A}}(\bn') g(A)\Big\}\, d S_{\bn'}
=\int_{\mathcal A}\Big\{\sum_{A\in
    Y(\nn,\bn')} g(A)\Big\}\, d S_{\bn'},
\end{multline*}
which proves  the assertion.
\end{proof}

\subsection{Regular Points and their Properties}

The following lemma shows that, for fixed $\nn \in
C^\infty(D_1, \SS)$, the set $Y(\nn,\bn')$ is well behaved  for most
$\bn'\in \mathbb S^2$. Recall that $D_1^\circ$ is the interior of $D_1$.

\begin{lemma}\label{krit6}
For  $N>1$, $N \in \mathbb N$, there is a compact set $\mathcal
Q_N\subset \mathbb S^2$ such that:
\begin{itemize}\setlength\itemsep{0em}
\item[$(a)$]    $\bn'\in \mathcal Q_N$ implies that
 $Y(\nn, \bn') \subset D_1^\circ$, $\rm {card}(Y(\nn,\bn')) \leq N$,
\begin{equation}\label{krit7}
\int_{D_1}\frac{dX}{|\mathbf n(X)-\bn'|}\leq N \text{ and }~
|\bn'\pm \boldsymbol k|\geq \frac{1}{N}.
\end{equation}
\item[$(b)$]  Each $A\in Y(\nn,\bn')$, $\bn'\in \mathcal Q_N$, is
non-degenerate, meaning that $\Phi(A)\neq 0$, equivalently $\det
D\nn(A) \neq 0$, or $\partial_i\mathbf n(A),i=1,2$, are linearly
independent.
\item[$(c)$] For an absolute constant $c$,
\begin{equation*}\label{krit8}
    \text{\rm meas }(\mathbb S^2\setminus \mathcal Q_N)\,\leq\,
    \frac{c}{N}\Big(\int_{D_1}|\nabla \mathbf n|^2\, dX+1\Big).
\end{equation*}
\end{itemize}
\end{lemma}
\begin{proof} The set $R_0=\{\bn' \in \SS: Y(\nn,\bn') \cap \partial D_1\neq \emptyset\}\subset  \nn(\partial D_1)$  has zero measure in $\SS$, because $\partial D_1$ has zero measure in $\RR^2$ and  $\nn:D_1 \to \SS$ is smooth.  From  \eqref{krit3} with $g\equiv 1$ and \eqref{aelita18},
$$
\int_{\mathbb S^2} \rm{card } (Y(\nn,\bn'))\,
dS_{\bn'}=\int_{D_1}|\Phi(X)|\, dX\leq \frac{1}{2}\int_{D_1}|\nabla
\mathbf n(X)|^2\, dX = \frac{1}{2} E(\nn),
$$ whence
\begin{equation}\label{krit9}
    \text{meas}\, (R_1)\leq \frac{E(\nn)}{2N} \text{ where $R_1 = \{\bn' \in
 {\SS}: \rm{card\,} (Y(\nn,\bn'))>
N\}$.}
\end{equation}
Since, by Fubini's theorem and \eqref{pint},
$$
\int_{\mathbb S^2}\Big\{\int_{D_1}\frac{dX}{|\mathbf n(X)-
\bn'|}\Big\}dS_{\bn'}=\int_{D_1}\Big\{\int_{\mathbb
S^2}\frac{1}{|\mathbf n(X)-\bn'|}dS_{\bn'} \Big\}dX\leq 4\pi^2,
$$
 it follows that
\begin{equation}
    \text{meas}\, (R_2)\leq
    \frac{4\pi^2}{N} \text{ where }
R_2= \left\{\bn' \in  {\SS}: \int_{D_1}\frac{dX}{|\mathbf
n(X)-\bn'|}> N\right\}.
\end{equation}
Note also from \eqref{pint} that
\begin{equation}\label{krit11}
    \text{meas}\, (R_3^\pm)\leq
    \frac{4\pi}{N}, \text{ where $R^\pm_3 =\{\bn'\in  {\SS}: |\bn'\pm k|< 1/N \}$.}
\end{equation}
Finally it follows from \eqref{krit9}-\eqref{krit11} that
\begin{equation*}\label{krit12}
    \text{meas}\, (R^*)\leq
     \frac{E(\nn)}{2N}+\frac{4\pi^2}{N}+\frac{8\pi}{N} \text{ where  }R^*=R_0\cup R_1\cup R_2\cup R_3^+ \cup R_3^-.
\end{equation*}
Hence there is an open  set $O_N\supset R^*$ such that, for an
absolute constant $c$,
\begin{equation*}
\text{meas}\, (O_N)\leq
     \frac{c}{N}\big(E(\nn) +1\big),
\end{equation*}
and $\mathcal Q_N = \SS \setminus O_N$
 satisfies parts ($a$) and ($c$).

 To prove ($b$),  suppose that some $A  \in Y(\nn,\bn')$ is degenerate, i.e., $\mathbf n(A) = \bn'$ and  $\partial
_i\mathbf n(A)$, $i=1,2$, are linearly dependent. Since the mapping
$\mathbf n(X)$ is infinitely differentiable,
$$
\partial_1\mathbf n(A)=\alpha \partial_2\mathbf n(A)\text{~or~}
\partial_2\mathbf n(A)=\alpha \partial_1\mathbf n(A)
$$
for some constant $\alpha$. In the first case (the second is
similar) for $X \in D_1$,
\begin{align*}
|\mathbf n(X)-\bn'|&=\big|\big(\alpha
(X_1-A_{1})+(X_2-A_{2})\big)\partial_2\mathbf n(X)+O(|X-A|^2)\big|
\\
& \leq c|\alpha (X_1-A_{1})+(X_2-A_{2})|+ c|X-A|^2,
\end{align*}
 where $c>0$ is some constant. Hence
\begin{equation*}\begin{split}
\int_{D_1}\frac{dX}{|\mathbf n(X)-\bn'|}
=\infty,
\end{split}\end{equation*}
which contradict \eqref{krit7}. This completes the proof of Lemma
\ref{krit6}.
\end{proof}
If  $Y(\nn,\bn') \neq \emptyset$ for $\nn \in C^\infty(D_1, \SS)$ and   $ \bn' \in
\mathcal Q_N$, let $$ Y(\nn,\bn') =
\{A_1(\bn'), \cdots, A_n(\bn')\},~A_i(\bn') \neq A_j(\bn'),\,i\neq
j,  \text{ where }n \leq N. $$

\begin{lemma}\label{krit13}  There exists $r_N>0$  such that, for  $\bs n'\in \mathcal Q_N$,
\begin{equation}\label{krit14}
|A_i(\bn')| < 1- r_N \text{ and }|A_i(\bn')-A_j(\bn')| > r_N,\quad 1\leq  i < j\leq n.
\end{equation}

\end{lemma}
\begin{proof} Suppose no $r_N>0$ satisfies the first inequality \eqref{krit14}. Then, since $\mathcal Q_N$ and $D_1$ are compact, there exist  $i\in \{1,\cdots, N\}$ and $\{\bn_k'\}\subset \mathcal Q_N$ with $\bn'_{k} \to \bn' \in \mathcal Q_N$ and $A_{i}(\bn_k') \to A$ where $|A| = 1$. Since $\nn(A_{i}(\bn_k')) = \bn_k'$, it follows that $\nn(A) = \bn'$. But $A \in \partial D_1$ and $\bn' \in \mathcal Q_N$ is false, by Lemma \ref{krit6}.

 If the second inequality  \eqref{krit14} is false, there exist  $\{i_{_k}\}, \{j_{_k}\}$, and
$\{\bn_k'\}$ such that
\begin{equation}\label{krit15}
 i_{_k}\neq
j_{_k}, \quad \bn_k'\in \mathcal Q_N\text{ and
}A_{i_{_k}}(\bn_k')-A_{j_{_k}}(\bn_k')\to 0\text{  as }k\to \infty.
\end{equation}
Taking subsequence if necessary, assume  $A_{i_{_k}}(\bn_k')\to A_i$
and $\bn_k'\to \bn'$ as $k\to \infty$. Then $\mathbf n(A_i)=\bn'$,
since $\nn$ is continuous,  and $\bn'\in \mathcal Q_N$, since
$\mathcal Q_N$ is compact.

Since the mapping $X \mapsto \mathbf n(X)$ is non-degenerate at $X=
A_i$, the derivative $D\mathbf n(A_i):\mathbb R^2\to T_{\bn'}
(\mathbb S^2)$ has a bounded inverse with $\|D\mathbf
n(A_i)^{-1}\|\neq 0$. Hence, since $\mathbf n:X \to \mathbb S^2$ is
smooth, there exists $\rho>0$ such that
\begin{equation}\label{rhoo}
\|D\mathbf n(X)-D\mathbf n(A_i)\|\leq\frac{1}{3}\|D\mathbf n
(A_i)^{-1}\|^{-1} \text{~for~}|X-A_i|\leq \rho.
\end{equation}
So, for  $X',X''\in D_\rho(A_i)$, the disc  of radius $\rho$ about
$A_i$,
\begin{align*}
&\mathbf n(X'')-\mathbf n(X')=\int_0^1 D\mathbf
n\big(X'+t(X''-X')\big)(X''-X')\,dt\\&= D\mathbf
n\big(A_i\big)(X''-X')+ \int_0^1  \big\{ D\mathbf
n\big(X'+t(X''-X')\big)-D\mathbf n\big(A_i\big)\big\}(X''-X')\,dt.
\end{align*}
Since $X'+t(X''-X')\in D_\rho(A_i)$ it follows from \eqref{rhoo}
that
\begin{align*}
|\mathbf n(X'')-\mathbf n(X')|&\geq \frac{|X'-X''|}{\|D\mathbf
n(A_i)^{-1}\|}\\&- \sup\limits_{t\in [0,1]} \big\{ \|D\mathbf
n(X'+t(X''-X'))-D\mathbf n(A_i)\big\}\|\,|X''-X'|\\&
\geq\frac{2}{3\|D\mathbf n(A_i)^{-1}\|}|X'-X''|=c|X''-X'|, \quad
c>0, \text{ say}.
\end{align*}
Therefore since $\mathbf n(A_{i_{_k}})=\mathbf n(A_{j_{_k}})=\bn_k'$
for all $k$ and, by \eqref{krit15}, for $k$ sufficiently large
$A_{i_{_k}}(\bn'), A_{j_{_k}}(\bn') \in D_\rho(A_i)$,
$$
|A_{i_{_k}}(\bn')-A_{j_{_k}}(\bn')|\leq \frac{1}{c}|\bn_k'-\bn_k'|=0
\text{  for all sufficiently large $k$}.
$$
But $A_{i_{_k}}(\bn')\neq A_{j_{_k}}(\bn')$ for all $k$. This
contradiction completes the proof.
\end{proof}

\begin{lemma}\label{krit16} For $\bn_0'\in \mathcal Q_N$ and  $\delta>0$, let
$ \mathcal W_\delta(\bn_0')=\{ \bn\in \mathbb S^2: \,
|\bn_0'-\bn'|<\delta\}. $ Then there exists $\delta>0$ such that
$\rm {card\,}(Y(\nn,\bn'))=\rm{ card\,}(Y(\nn,\bn_0'))$ for all
$\bn'\in \mathcal W_\delta (\mathbf n_0')$. Moreover, $Y(\nn,\bn')$
can be labelled  $A_\ell(\bn'),\,1 \leq \ell\leq n\leq N$, where
\begin{equation*}\label{krit17}
A_\ell\in C^\infty(\mathcal W_\delta(\bn_0')), \quad A_\ell(\bn')\to
A_\ell(\bn_0')\text{~as~} \bn'\to \bn_0'.\end{equation*}
\end{lemma}
\begin{proof} Let $n = \rm
{card\,}(Y(\nn,\bn_0'))$ where $Y(\nn,\bn_0') = \{A_\ell(\bn_0'): 1
\leq i\leq n\}$, and choose an arbitrary $\ell\in \{1,\cdots, n\}$.
Since the mapping $X \mapsto \mathbf n(X)$ is non-degenerate at $X
=A_\ell(\bn'_0)$, its derivative $D\mathbf n(A_\ell(\bn'_0)):\mathbb
R^2\to T_{\bn_0'} (\mathbb S^2)$, has  bounded inverse. Hence, by
the inverse function theorem, there exist $\delta_\ell>0$ and
$\rho_\ell\in (0, r_N)$ (see Lemma \ref{krit13}) such that the equation $\mathbf n(X)=\bn'$
has a unique solution $X=A_\ell(\bn')$ in the disc
$D_{\rho_\ell}(A_\ell(\bn'_0))$ for every $\bn'\in \mathcal
W_{\delta_\ell}(\bn_0')$. Moreover, the mappings $\bn' \mapsto
A_\ell(\bn')$ are infinitely differentiable on $\mathcal
W_{\delta}(\bn'_0)$, where $\delta=\min\{\delta_\ell: 1 \leq
\ell\leq n\}$, and $A_\ell(\bn')\to A_\ell$ as $\bn'\to \bn_0'$ for
all $1 \leq \ell \leq n$.

It remains to show that, for sufficiently small $\delta$, all the
solutions $X \in D_1$  of the equation $\mathbf n(X)=\bn',\,\bn'\in
\mathcal W_\delta(\bn'_0)$,  are given by
 $\{A_\ell(\bn'), \,1\leq i\leq
 n\}$. If no such $\delta>0$ exists, then there are sequences
$$
\delta_k\to 0 \text{ in } \RR, \quad \bn_k'\to \bn_0' \text{ in }
\SS, \quad \mathbf n(B_k)=\bn_k', ~B_k \in D_1,
$$
such that for all $\ell$,
$$
\quad |B_k- A_\ell(\bn_k')|>\rho_\ell ~\text{ where }~ A_\ell(
\bn_k')\to A_\ell(\bn_0')\text{~as~} k\to \infty.
$$
Since $\rho_\ell>0$ is independent of  $k$, after passing to a
subsequence if necessary, $B_k\to B$ as $k\to \infty$ and hence
$\mathbf n(B)=\bn_0'$ and $|B-A_\ell(\bn_0')|\geq \rho_\ell,\,1\leq
\ell \leq n_0 $. This is false since $A_\ell(\bn_0'),\,1\leq \ell
\leq n_0$, are the only solutions of $\mathbf n(X)=\bn_0'$.
\end{proof}

When  $\bn'\in \mathcal Q_N$, the next two lemmas concern the
regularity of the functions $ \omega_i(\nn, \bn')$, $i=1,2$, which
were defined   in \eqref{leila25} by
\begin{multline*}
\frac{m_3(X)+1}{m_1(X)^2+m_2(X)^2}
    \Big(m_1(X)(\mathbf U(\bn')\partial_i\mathbf n(X))_2-m_2(X)(\mathbf U(\bn')\partial_i\mathbf n(X))_1\Big).
\end{multline*}
Here $\mathbf m(X)= (m_1(X),m_2(X),m_3(X)) = \mathbf U( \bn')\,
\mathbf n(X), $ and the rotation matrix $\mathbf U(\bn')$ is given
by \eqref{leila15}. The notation  $\mathcal Q_N\subset \SS$, $r_N>0$
and $Y(\nn,\bn_0')$ are as in Lemma \ref{krit13} and \ref{krit16}.
\begin{lemma}\label{krit19}  For $\bn'_0 \in \mathcal Q_N$,
the functions  $\omega_i,\,i=1,2$, are infinitely differentiable with
respect to $(X,\bn')$ on $\big(D_1\setminus \bigcup_{\ell = 1}^n
D_{r_N/3}(A_\ell(\bn'_0)\big)\times \mathcal W_\delta(\bn_0')$, for
some $\delta>0$.
\end{lemma}
\begin{proof} Consider the compact sets
$$
F= D_1\setminus \bigcup_{\ell = 1}^n D^\circ_{r_N/3}(A_\ell(\bn'_0))
\text{ and }E=\mathbf n(F)\times \{\bn_0'\},
$$
where $D^\circ_{r}(A)$ is the open disc with centre $A \in D_1$ and
radius $r$. Since $\bn_0'\in \mathbb S^2\setminus \{\pm \boldsymbol
k\}$ and $\mathbf n(X)\neq \bn_0'$, $X \in F$, it follows that $E
\subset \Sigma$ (Definition \ref{leiladefinition}) is compact. An
application of Lemma \ref{leila60} with $K=\{\bs n_0'\}$ completes
the proof. \end{proof}

\begin{lemma}\label{krit20} For smooth $\nn$ and  any $\bn_0' \in \mathcal Q_N$, the functions $\omega_i(\cdot, \bn_0'),\,i=1,2,$ are
infinitely differentiable on  $D_1\setminus
\bigcup_{\ell=1}^n\{A_\ell(\bn_0')\}$, and 
\begin{equation}\label{krit21}
\Phi(X)=\partial_2\,\omega_1(X, \bn_0')\,\big)-\partial
    _1\omega_2(X, \bn_0'),\quad
    |\omega_i(X, \bn_0')|\leq c\,\frac{|\partial_i \mathbf n(X)|}{|\mathbf n(X)-\bn_0'|}
\end{equation}
 for an absolute constant
$c$.
In particular, $\omega_i(\cdot, \bn_0')$ is integrable on
 $D_1$ and
 \begin{equation}\label{krit22}
    \int_{D_1}|\omega_i(X, \bn_0')|\,dX\leq c\,N \|\mathbf
    n\|_{C^1(D_1)}
\end{equation}
\end{lemma}
\begin{proof}  Lemma \ref{leila26} yields  \eqref{krit21}, and  \eqref{krit22} follows from \eqref{krit7} and \eqref{krit21}.
\end{proof}

\subsection{More on $\boldsymbol{\omega_i,\,i=1,2}$}
 For fixed $N$ let $r_N>0$ be given by Lemma \ref{krit13} and, for $\bn' \in \mathcal Q_N$, let
$ Y(\nn,\bn_0') = \{A_\ell(\bn'_0):1\leq \ell\leq n\}\}$. Then for
$\zeta\in C^\infty_0 (D_1)$ it follows from Lemma \ref{krit20} and
the divergence theorem that, for every $\bn'\in \mathcal Q_N$ and $r
\in (0, r_N)$,
\begin{equation}\label{krit23}
    \int_{D_1\setminus \cup_{\ell = 1}^n D_r(A_\ell(\bn'))}\Phi(X) \zeta(X)\, dX =
    I(r, \bn')+J(r, \bn'),\end{equation}
  where   \begin{align*}
I(r, \bn')&= \int_{D_1\setminus \cup_{\ell = 1}^n
D_r(A_\ell(\bn'))}\big(\,\partial_1\zeta\,
\omega_2(X, \bn')-\partial_2\zeta\, \omega_1(X, \bn')\,\big)\, dX,\\
J(r, \bn')&= \sum_\ell^n\int_{\partial D_r(A_\ell(\bn'))}\zeta\,
\big(\,\omega_2(X, \bn')\nu_1- \omega_1(X, \bn')\nu_2\,\big)\, ds,
\notag\end{align*} and $(\nu_1,\nu_2)$ is the outward normal to
$\partial D_r(A_\ell(\bn'))$.

The proof of Theorem \ref{aelita30} depends on the following
technical lemma and the calculation in Lemma \ref{krit25} which
follows it.

\begin{lemma}\label{krit31}
$I(r,\bn')$ and $J(r,\bn')$ are continuous with respect to $\bn'$ on
$\mathcal Q_N$.

\end{lemma}

\begin{proof} Fix  $\bn_0'\in \mathcal Q_N$. Then  $\omega_i,\,i=1,2,$ are
infinitely differentiable on  $\big(D_1\setminus \bigcup_\ell^n
D_{r/3}(A_\ell(\bn'_0))\big)\times \mathcal W_{\delta}(\bn_0')$,
by Lemma \ref{krit19}. Moreover, by Lemma \ref{krit16}, for some
$\delta>0$ the function $A_\ell \in C^\infty(\mathcal
W_\delta(\bn_0'))$ and  $A_\ell(\bn')\to A_\ell(\bn_0')$ as $\bn'\to
\bn_0'$ in $\mathcal Q_N$. Therefore, for $\delta>0$ sufficiently
small, $|A_\ell(\bn')-A_\ell(\bn'_0)| <r/3$ for all $\bn'\in
\mathcal W_\delta(\bn_0')$. It follows that $\partial
D_r(A_\ell(\bn'_0))\subset D_1\setminus \cup_{\ell=1}^n
D_{r/3}(A_\ell(\bn'_0))$ for all such $\bn'$. Changing coordinates
$X \mapsto X+A_\ell(\bn')-A_\ell(\bn'_0)$ yields, for all $1\leq
i\leq n$ and $j=1,2$, that
\begin{multline*}
\int_{\partial D_r(A_\ell(\bn'))} \zeta(X)\,\omega_i(X, \bn')\nu_j\,
ds\\=\int_{\partial
D_r(A_\ell(\bn'_0))}\zeta(X+A_\ell(\bn')-A_\ell(\bn'_0))\,\omega_i(X+A_\ell(\bn')-A_\ell(\bn'_0),
\bn')\nu_j ds.
\end{multline*}
Since $X+A_\ell(\bn')-A_\ell(\bn'_0)\in D_1\setminus \bigcup_\ell
D_{r/3}(A_\ell(\bn'_0))\big)$, the functions $(X,\bn') \mapsto
\omega_i(X+A_\ell(\bn')-A_\ell(\bn'_0), \bn')$ are continuous on
$\partial D_r(A_\ell(\bn'_0))\times \mathcal W_\delta(\bn_0')$,
whence
\begin{multline*}
\int_{\partial
D_r(A_\ell(\bn'_0))}\zeta(X+A_\ell(\bn')-A_\ell(\bn'_0)\,\omega_i(X+A_\ell(\bn')-A_\ell(\bn'_0),
\bn'))\nu_j ds\\ \to \int_{\partial D_r(A_\ell(\bn'_0))}
\zeta(X)\,\omega_i(X, \bn_0')\nu_j\, ds \text{ as $\bn'\to \bn_0'$.}
\end{multline*}
 Therefore  $\bn' \mapsto J(r,
\bn')$ is continuous at every $\bn'_0\in \mathcal Q_N$.

Next set $G(r, \bn')=D_1\setminus \cup_{\ell=1}^n
D_r(A_\ell(\bn'))$. Then
\begin{multline*}
|I(r, \bn')-I(r, \bn_0')|\leq c\int_{G(r,\bn') \cap
G(r,\bn_0')}|\omega_i(X,\bn')-\omega_i(X,
\bn_0')|\, dX\\
+c\int_{G(r,\bn') \setminus G(r,\bn_0')}|\omega_i(X, \bn_0')|\,
dX+c\int_{G(r,\bn_0') \setminus \mathcal
G(r,\bn')}|\omega_i(X,\bn')|\, dX.
\end{multline*}
Since, for $|\bn'-\bn_0'|\geq r/3$, the closure of $G(r, \bn')\cup
G(r, \bn_0')$ is a compact subset of $ D_1\setminus\cup_{\ell =1}^n
D_{r/3}(A_\ell(\bn'_0))$, the functions $\omega_i(X, \bn')$ are
uniformly continuous and uniformly bounded on
 $G(r, \bn')\cup G(r, \bn_0')\times \mathcal W_\delta(\bn_0')$. Moreover
$$
\text{meas~}G(r,\bn') \setminus G(r,\bn_0')+ \text{meas~}
G(r,\bn_0') \setminus G(r,\bn')\to 0
$$
as $\bn'\to\bn_0'$. It follows that $I(r, \bn')$ is continuous at
every point $\bn_0'\in \mathcal Q_N$.
\end{proof}
\begin{lemma}\label{krit25} For $\bn' \in \mathcal Q_N$ and
    $Y(\nn,\bn')=\{A_\ell(\bn'): 1 \leq i \leq n\}$,
\begin{align}
\label{krit27}
    \lim\limits_{r\to 0} I(r, \bn')&=\int_{D_1}\big(\,\partial_1\zeta \omega_2(X, \bn')-\partial_2\zeta
\omega_1(X, \bn')\,\big)\, dX, \\
\label{krit26}
    \lim\limits_{r\to 0} J(r, \bn')&=4\pi\sum_{\ell=1}^n  \zeta(A_\ell(\bn'))\text{\rm sign}\, \Phi(A_\ell(\bn')).
\end{align}
\end{lemma}
\begin{proof}
To prove \eqref{krit27} it is suffices to note from \eqref{krit7}
and Lemma \ref{krit20} that the functions $\omega_i(\cdot,
\bn'),\,i=1,2$, are  integrable in $D_1$ for every $\bn'\in \mathcal
Q_N$.

The proof of \eqref{krit26} is more complicated. Fix an arbitrary
$\bn'\in \mathcal Q_N$ and recall from \eqref{leila25} that
\begin{equation*}\label{krit32}
\omega_i(X, \bn')=(m_3+1)\frac{m_1\partial_i m_2-m_2\partial_i
m_1}{m_1^2+m_2^2},
\end{equation*}
where $\mathbf m(X)=\mathbf U(\bn') \mathbf n(X)$ and the orthogonal
matrix $\mathbf U(\bn')$ is defined by \eqref{leila15}. Therefore $
\mathbf m(A_\ell(\bn'))= \mathbf U(\bn')\bn'=\boldsymbol k =(0,0,1)
$ for $A_\ell(\bn')\in Y(\nn,\bn').$

Now fix $\ell\in \{1,\cdots, n\}$ and to simplify notation change
the origin of coordinates in $D_1$ so that $A_\ell(\bn')=0$,
$\mathbf m(0)=\boldsymbol k$ and, as $X \to 0$,
\begin{equation*}
    \mathbf m (X)=\boldsymbol k +X_1\boldsymbol\mu_1+X_2\boldsymbol \mu_2 +O(|X|^2), \text{ where }\boldsymbol \mu_i = \mathbf U(\bn') \partial_i\mathbf n(0) \in \mathbb R^3.
\end{equation*}
Since $\boldsymbol \mu_i\perp \boldsymbol k$, because $\mathbf
m(X)\in \SS$,
 it follows that
\begin{equation}\label{krit33}
    (m_1(X),m_2(X))=X_1\boldsymbol \mu_1+X_2\boldsymbol \mu_2+\mathbf h(X), \quad
    |\mathbf h(X)|\leq c  |X|^2.
\end{equation}

With $A_\ell(\bn') = 0$, it follows from Lemma \ref{krit6} (b) that
$\partial_i\mathbf n(0)$,  and hence  $\bs\mu_i $, $i=1,2$, are
linearly independent. Therefore
$$
\sqrt{m_1^2+m_2^2}\geq c_1|X|-c_2|X|^2\geq c r \text{~on~}\partial
D_r (0),
$$
for constants $c_1,c_2$, and $r>0$ sufficiently small. It follows
that
$$
\frac{|m_1\partial_i m_2-m_2\partial_i m_1|}{m_1^2+m_2^2}\leq \frac
cr\text{~on~}\partial D_r.
$$
Since $(1-m_3)(1+m_3)=m_1^2+m_2^2$ and $m_3(0)=1$,
$$
|(m_3+1)-2|\leq c (m_1^2+m_2^2)\leq cr^2\text{~on~} \partial D_r
$$
for all sufficiently small $r$. Therefore
$$
\int_{\partial
D_r}\Big|\zeta(X)(m_3+1)-2\zeta(0)\Big|\frac{|m_1\partial_i
m_2-m_2\partial_i m_1|}{m_1^2+m_2^2}ds\leq c\int_{\partial
D_r}ds\leq c r\to 0
$$
as $r\to 0$. Hence, similarly, for $i,j = 1,2$, $i \neq j$,
\begin{equation}\label{krit34}
    \lim\limits_{r\to 0}\int_{\partial D_r}\zeta(X)\omega_i(X,
    \bn')\nu_jds=2\zeta(0)\lim\limits_{r\to 0}
   \int_{\partial D_r} \nu_j\frac{m_1\partial_i m_2-m_2\partial_i
m_1}{m_1^2+m_2^2}ds.
\end{equation}
Since $\boldsymbol \mu_i\perp \boldsymbol k$ in $\mathbb R^3$,  a
linear function $\mathbf p:\mathbb R^2 \to \mathbb R^2$ is defined
by putting
$$
\mathbf p(X)= (p_1(X), p_2(X))\text{ where } (p_1(X), p_2(X),0)=X_1
\boldsymbol \mu_1+X_2\boldsymbol \mu_2.
$$
It follows from \eqref{krit33} that
$$
\big|(p_1(X), p_2(X))- (m_1(X), m_2(X))\big|\leq c|X|^2,
$$
and hence
$$
\Big|\frac{m_1\partial_i m_2-m_2\partial_i m_1}{m_1^2+m_2^2}-
\frac{p_1\partial_i p_2-p_2\partial_i p_1}{p_1^2+p_2^2}\Big|\leq c
$$
on $\partial D_r$ for $r$ sufficiently small. It follows that
\begin{equation}\label{krit35}
\lim\limits_{r\to 0} \int_{\partial D_r} \nu_j\frac{m_1\partial_i
m_2-m_2\partial_i m_1}{m_1^2+m_2^2}ds=\lim\limits_{r\to 0}
\int_{\partial D_r} \nu_j\frac{p_1\partial_i p_2-p_2\partial_i
p_1}{p_1^2+p_2^2}ds.
\end{equation}
Since the integral on the right is independent of $r$, it can be
replaced by the same integral over $\partial D_1$, and so, combining
\eqref{krit34} with \eqref{krit35}, yields
\begin{multline}\label{krit36}
 \lim\limits_{r\to 0}\int_{\partial D_r}\zeta(X)(\omega_2(X,
    \bn')\nu_1-\omega_2(X,
    \bn')\nu_1)ds\\=2\zeta(0)
\int_{\partial D_1}\Big\{\frac{p_1\partial_2 p_2-p_2\partial_2
p_1}{p_1^2+p_2^2}\nu_1-\frac{p_1\partial_1 p_2-p_2\partial_1
p_1}{p_1^2+p_2^2}\nu_2\Big\}\, ds.
\end{multline}
It remains to calculate the right side of \eqref{krit36}. Since
$\boldsymbol \mu_i,\,i=1,2$, are linearly independent,  $\mathbf
p(\partial D_1)$ is a strictly convex curve with the origin as an
interior point and a  continuous function
$$
\Theta(X)=\text{rag~}(p_1(X)+i p_2(X))
$$
can be defined on $\partial D_1$. Now define polar coordinates on
$\partial D_1$ by $X_1=\cos(\theta+\theta_0)$,
$X_2=\sin(\theta+\theta_0)$  ($\theta_0$ to be chosen later) and let
$\Psi(\theta)=\Theta(X(\theta))$. Then
\begin{multline*}
\frac{p_1\partial_2 p_2-p_2\partial_2 p_1}
{p_1^2+p_2^2}\nu_1-\frac{p_1\partial_1 p_2-p_2\partial_1
p_1}{p_1^2+p_2^2}\nu_2 \\= \frac{p_1\partial_2 p_2-p_2\partial_2
p_1}{p_1^2+p_2^2}\cos(\theta+\theta_0)-\frac{p_1\partial_1
p_2-p_2\partial_1 p_1}{p_1^2+p_2^2}\sin (\theta+\theta_0) =
\partial_\theta \Psi(\theta).\end{multline*} Therefore, since $ds =
d\theta$,
 it follows from  \eqref{krit36} that
\begin{equation}\label{krit37}\begin{split}
 \lim\limits_{r\to 0}\int_{\partial D_r}\zeta(X)(\omega_2(X,
    \bn')\nu_1-\omega_2(X,
    \bn')\nu_1)ds=2\zeta(0)
\int_{[0,2\pi)} \partial_\theta \Psi(\theta) d\theta.
\end{split}\end{equation}
To study  $\Psi$, for $i=1,2$ let
$$
\boldsymbol \lambda_i=(\mu_{1i}, \mu_{2i}):=\rho_i(\cos\beta_i, \sin
\beta_i), \text{ where } (\mu_{i1},\mu_{i2}) = \boldsymbol
\mu_i,~\beta_i \in [0,2\pi).
$$
Since the $\boldsymbol \mu_i$ are linearly independent, so are the
$\boldsymbol \lambda_i$ and $\rho_i>0$. Hence
$$
p_i(X(\theta))= \boldsymbol \lambda_i\cdot X(\theta)=
\lambda_{i1}\cos(\theta+\theta_0)+\lambda_{i2}\sin(\theta+\theta_0)=
\rho_i \cos(\theta+\theta_0-\beta_i).
$$
Now set $\theta_0=\beta_1$ to obtain
\begin{equation*}\label{krit38}
\Psi(\theta)=\text{rag~}(\rho_1\cos\theta+i\rho_2\cos(\theta-\alpha)),
\quad \alpha=\beta_2-\beta_1 \neq 0.
\end{equation*}
To calculate $ \int_{[0, 2\pi)} \partial_\theta \Psi(\theta) d\theta
= \lim\limits_{\theta{\nearrow2\pi}}\Psi(\theta)-\Psi(0) $ in
\eqref{krit37} let
$$
b(\theta)=\frac{\rho_2\cos(\theta-\alpha)}{\rho_1\cos\theta}=
\frac{\rho_2}{\rho_1}\big(\cos \alpha + \tan \theta \sin
\alpha\big).
$$
Since $b(0)= \tan\Psi(0)$ and  $\lim\limits_{\theta\nearrow
\pi/2}b(\theta)= \pm\infty$ when  $\pm \sin\alpha>0$,
\begin{equation}\label{krit39}
\lim\limits_{\theta \nearrow \pi/2} \Psi(\theta)-\Psi(0)=
\text{sign~}(\sin\alpha) \frac{\pi}{2}-\arctan b(0).
\end{equation}
Similarly
\begin{align*}
&\lim\limits_{\theta\searrow\pi/2}b(\theta)= \mp\infty \text{~if~}
\pm\sin\alpha>0\text{~and~} \lim\limits_{\theta\nearrow
3\pi/2}b(\theta)= \pm\infty \text{~if~} \pm\sin\alpha>0,
\\
\intertext{and hence} &\lim\limits_{\theta\searrow \pi/2} \arctan
b(\theta)=-\text{sign~}(\sin{\alpha})\frac{\pi}{2} \text{ and }
\lim\limits_{\theta\nearrow 3\pi/2} \arctan
b(\theta)=\text{sign~}(\sin{\alpha})\frac{\pi}{2}.
\end{align*}
It follows
\begin{equation*}\label{krit40}
\lim\limits_{\theta\nearrow 3\pi/2}
\Psi(\theta)-\lim\limits_{\theta\searrow \pi/2}
\Psi(\theta)-\Psi(0)= \text{sign~}(\sin\alpha) \pi.
\end{equation*}
Finally,
$$
\lim\limits_{\theta\searrow 3\pi/2}b(\theta)= \mp\infty \text{~if~}
\pm\sin\alpha>0
$$

 yields
$$\lim\limits_{\theta\searrow 3\pi/2} \arctan
b(\theta)=-\text{sign}(\sin{\alpha})\pi/2,
$$
and
\begin{equation}\label{krit41}
\Psi(2\pi) -\lim\limits_{\theta\searrow 3\pi/2} \Psi(\theta)=
\text{sign~}(\sin\alpha) \pi/2+ \arctan b(0).
\end{equation}
Since $\Psi(\theta)$ is continuous on $[0,2\pi)$,
\eqref{krit39}-\eqref{krit41} lead to
\begin{equation}\label{krit44}\begin{split}
\int\limits_{[0, 2\pi)} \partial_\theta \Psi(\theta)
d\theta=\lim\limits_{\theta\to {2\pi-0}}\Psi(\theta)-A(0)
=\text{sign~}(\sin\alpha) 2\pi.
\end{split}\end{equation}
Since
$$\sin\alpha = \sin (\beta_2-\beta_1)=
\det\left(
\begin{array}{ccc}
0 &0 & 1
\\
\cos\beta_1 & \sin \beta_1 &0
\\
\cos \beta_2 &\sin\beta_2&0
\end{array}\right)
    $$
the signum of $\sin \alpha$ coincides with  the orientation of the
triple $(\boldsymbol k,\boldsymbol \mu_1, \boldsymbol \mu_2)$.
Therefore, since
$$
\boldsymbol k=\mathbf U(\bn')\bn', \quad \boldsymbol \mu_i= \mathbf
U(\bn')\partial_i\mathbf n(0),
$$
and $\det\mathbf U(\bn')=1$,  the orientation of the triplet
 $(\boldsymbol k,\boldsymbol \mu_1, \boldsymbol \mu_2)$ is the same that of $(\mathbf n(0),\partial_1\mathbf n(0), \partial_1\mathbf
 n(0))$ which equals $\text{sign}\,
\big(\mathbf n(0)\cdot\partial_1\mathbf n(0)\times \partial_2
\mathbf n(0)\big)=\text{sign}\, \Phi(A_\ell(\bn'))$.  It follows that
$\text{sign~}(\sin\alpha)=\text{sign~}\Phi(A_\ell(\bn'))$. Combining
these result with \eqref{krit37} and \eqref{krit44} and recalling
that $A_\ell(\bn')=0$ leads to the identity
\begin{align*}
 \lim\limits_{r\to 0}\int_{\partial D_r(A_\ell(\bn'))}\zeta(X)(\omega_2(X,
    \bn')\nu_1&-\omega_2(X,
    \bn')\nu_1)ds\\&=4\pi \zeta(A_\ell(\bn'))\text{~sign}\,\Phi(A_\ell(\bn')),
\end{align*}
 where $Y(\nn,\bn') = \{A_1(\nn,\bn'),\cdots A_n(\nn,\bn')\}$, from which \eqref{krit26} follows.
\end{proof}
\subsection{Proof of Theorem
\ref{aelita30}} Since by Lemma \ref{krit31},  $I(r,\bn')$, $J(r, \bn')$ are
continuous in $\bn'$ on $\mathcal Q_N$, their limits,
\begin{equation}\begin{split}\label{krit45}
I(0, \bn')&= \quad\int_{D_1}\big(\,\partial_1\zeta
\omega_2(X, \bn')-\partial_2\zeta \omega_1(X, \bn')\,\big)\, dX,\\
J(0, \bn')&=4\pi\sum_{\ell}^n  \zeta(A_\ell(\bn'))\text{\rm sign}\,
\Phi(A_\ell(\bn')),
\end{split}\end{equation}
are measurable with respect to $\bn'\in \mathcal Q_N$. By Lemma
\ref{krit20}, $|I({0, \bn'})|$ is bounded, and hence integrable, on
$\mathcal Q_N$ and, since  the measurable function
$\zeta\,\text{sign~}\Phi$ is bounded, it follows from  Theorem
\ref{krit2} that  $J(0, \bn')$ is integrable over $\mathcal Q_N$.

Now for a Borel set $\mathcal A\subset \mathbb S^2$ of positive
measure, let
\begin{equation*}
 \mathcal A_N= \mathcal A\cap \mathcal Q_N,\quad
 \mathcal F=\nn^{-1}(\mathcal A), \quad \mathcal
    F_N=\nn^{-1}(\mathcal A_N),
\end{equation*}
and note from Lemma \ref{krit6}(c) that $\mathcal A_N$ has positive
measure for all $N$ sufficiently large. Now, for such $N$, let
$\Omega_{N,i}:D_1 \to \mathbb R$, $i=1,2$, be given by
\begin{subequations}\label{krit47}
\begin{align}
\Omega_{N,i}(X)=\frac{1}{\mu_N}\int_{\mathcal A_N}\omega_i(X, \bn')d
S_{\bn'},\text{~where~} \mu_N=\text{meas}\,
(\mathcal A_N),\\
\intertext{and note from Lemma \ref{krit20} that}
|\Omega_{N,i}(X)|\leq \frac{c}{\mu_N}\left(\int_{\mathcal A}\frac{d
S_{\bn'}}{ |\bn'-\mathbf n(X)|}\right)\,|\nabla\mathbf n(X)|, \quad
i= 1,2.
\end{align}
\end{subequations}
 Then \eqref{krit45} and Corollary
\ref{krit4} with $g=4\pi \zeta(X)\text{sign~}\Phi(X)$ and $\mathcal
A$ replaced by $\mathcal A_N$  imply that
\begin{equation}\label{krit48}
    \int_{\mathcal A_N} J(0, \bn')\, d S_{\bn'}=4\pi\int_{\mathcal
    F_N}\zeta \Phi\, dX.
\end{equation}
Letting $r\to 0$ in  \eqref{krit23}, it follows from Lemma
\ref{krit25} that
$$
 \int_{D_1}\Phi(X) \zeta(X)\, dX=
    I(0, \bn')+J(0, \bn').
$$
 Then integrating both  sides over  $\mathcal
A_N$ in the light of  \eqref{krit45},\eqref{krit47}
and\eqref{krit48} gives
\begin{equation}\label{krit50}
    \int_{D_1}\zeta\Phi\, dX=\frac{4\pi}{\mu_N}\int_{\mathcal F_N}
    \zeta\Phi\, dX+\int_{D_1}(\partial_1\zeta\, \Omega_{N,2}-\partial_2\zeta
    \,\Omega_{N,1})\, dX.
\end{equation}
Next, recall that $\mathcal Q_N \subset \mathcal
Q_{N+1}\subset\mathcal Q_\infty$  where $\mathcal Q_\infty =
\bigcup_N \mathcal Q_N$ and $ \text{meas}\,(\mathcal E)=0$ where $
\mathcal E=\mathbb S^2\setminus\mathcal Q_\infty. $ To let $N\to
\infty$ in \eqref{krit50}, note that $\mu_N\to
\mu=\text{meas}\,(\mathcal A)$,
\begin{equation}\label{krit52}
\int_{D_1}(\partial_1\zeta\, \Omega_{N,2}-\partial_2\zeta
    \,\Omega_{N,1})\, dX\to\int_{D_1}(\partial_1\zeta\, \Omega_{2}-\partial_2\zeta
    \,\Omega_{1})\, dX,
\end{equation}
where
$$
\Omega_i(X)=\frac{1}{\mu}\int_{\mathcal A}\omega_i(X,
    \bn')\, dS_{\bn'},
$$
and
\begin{equation*}\label{krit52a}
\int_{\mathcal F_N} \zeta\Phi\, dX\to \int_{\mathcal F_{\infty}}
\zeta\Phi\, dX=\int_{\mathcal F} \zeta\Phi\, dX-\int_{\mathcal
F\setminus\mathcal F_\infty} \zeta\Phi\, dX\text{~as~} N\to
\infty,
\end{equation*}
 where $\mathcal F_\infty=\nn^{-1}(\mathcal A\cap\mathcal
Q_\infty), \quad \mathcal F\setminus\mathcal
F_\infty=\nn^{-1}(\mathcal E\cap\mathcal A).$ To show that
\begin{equation}\label{krit52b}
\int_{\mathcal F\setminus\mathcal F_\infty} \zeta\Phi\, dX=0,
\end{equation}
let $\{G_n\}\subset \SS$ be a decreasing sequence of open sets such
that $\mathcal E \subset G_n$, $\text{meas}\, (G_n) \to
0\text{~as~} n\to \infty$ and let $F_n= \nn^{-1}(G_n)$. Then, since $\mathcal
F\setminus \mathcal F_\infty\subset F_n := \nn^{-1}(G_n)$, it
suffices to prove that
\begin{equation}\label{krit52c}
\int_{F_n} |\Phi|\, dX\to 0\text{~as~} n\to \infty.
\end{equation}
Now from Theorem \ref{krit2},
$$
\int_{\mathbb S^2}\text{card\,} Y(\nn,\bn')\, dS_{\bn'}
=\int_{D_1}|\Phi|\, dX<\infty,
$$
and from  Corollary \ref{krit4}, since $\text{meas}\,( G_n) \to 0$,
$$
\int_{F_n}|\Phi|\, dx =\int_{G_n}\text{card\,} Y(\nn,\bn')\,
dS_{\bn'}\to 0 \text{~as~} n\to\infty.
$$
  This yields \eqref{krit52c} (and
hence \eqref{krit52b}) and it follows that
\begin{equation}\label{krit53}
\int_{\mathcal F_N} \zeta\Phi\, dX\to \int_{\mathcal F} \zeta\Phi\,
dX\text{~as~} N\to \infty.
\end{equation}
Hence  \eqref{krit50}, \eqref{krit52} and \eqref{krit53} imply that
\begin{equation}\label{krit57a}
    \int_{D_1}\zeta\Phi\, dX=\frac{4\pi}{\mu}\int_{\mathcal F}
    \zeta\Phi\, dX+\int_{D_1}(\partial_1\zeta\, \Omega_{2}-
    \partial_2\zeta
    \,\Omega_{1})\, dX,
\end{equation}
where $ \mathcal F=\nn^{-1}(\mathcal A)$, $
\mu=\text{meas}\,(\mathcal A) $ and, by \eqref{krit47},
\begin{equation*}\label{54}
|\Omega_i(X)|\leq\frac{1}{\mu} c_A |\nabla \mathbf n(X)|, \quad
\|\Omega_i\|_{L^2(D_1)}\leq c_A\|\nabla\mathbf n\|_{L^2(D_1)},\quad
i=1,2,
\end{equation*}
where, with $c$ an absolute constant,
$$
c_{\mathcal A}=c\sup_{X \in D_1}\int_{\mathcal
A}\frac{dS_{\bn'}}{|\mathbf n(X)-\bn'|}.
$$
Let $\Sigma_\rho=\big\{\bn'\in \mathbb S^2: \text{~geodesic
distance~} (\bn', \boldsymbol k)\leq \rho\big\}, \text{ where }
\text{meas}\,(\Sigma_\rho)=\mu. $ Clearly  $c^{-1}\sqrt{\mu}\leq
\rho\leq c\sqrt{\mu}$ and
\begin{equation*}\begin{split}
\int_{\mathcal A}\frac{dS_{\bn'}}{|\mathbf n-\bn'|}\leq
\int_{\Sigma_\rho}\frac{dS_{\bn'}}{|\boldsymbol k-\bn'|}\leq
c\sqrt{\mu},
\end{split}\end{equation*}
where $c$ is an absolute constant. Thus
$$
|\Omega_i(X)|\leq \frac{c}{\sqrt{\mu}}\, |\nabla \mathbf n(X)|
\text{ and } \|\Omega_i\|_{L^2(D_1)}\leq \frac{c}{\sqrt{\mu}}\,
\|\nabla \mathbf n\|_{L^2(D_1)}.
$$
This and \eqref{krit57a} yield \eqref{aelita31}, \eqref{aelita32},
which completes proof of Theorem \ref{aelita30}. \qed

\section{Examples}\label{niza}

This section examines  two families of immersions with one
singularity, both of which show the hypotheses of Theorems
\ref{aelita22} (see Section \ref{r1}) and \ref{aelita30} (see
Section \ref{r2}) are optimal. The first family, in Section
\ref{enneper}, concerns scaled versions of Enneper's  minimal
surface,  in which singularity formation is associated with
self-intersection. The second example, in Section \ref{duality}, is
a family of  stereographic projection for which singularity
formation is due to bubbling. However the Gauss maps of  these two
immersions are the same, and hence the preceding theory applies
without change to both, although they arise as Gauss maps of
significantly different surfaces in $\RR^3$.

\subsection{Enneper's Surface}\label{enneper}
 Let $X=(X_1, X_2)$ be a coordinate in $\mathbb R^2$ and
$Z=X_1+iX_2\in \mathbb C$. The  zero-mean-curvature  Enneper surface
in $\RR^3$ is  defined in  parametric form by

\begin{equation*}\label{niza1}
 \mathfrak E=\left\{ \Psi(Z)= \frac{1}{2}\Re\Big\{ Z-\frac{1}{3} Z^3,\,\,
i(Z+\frac{1}{3}Z^3),\,\,Z^2\Big\}: \quad Z\in \mathbb C\right\}.
\end{equation*}
Here, as in \cite{kuwert}, consider the scaled version defined in
 complex form by
\begin{equation*}\label{niza2}
\Psi_\varepsilon(Z)\equiv
\frac{1}{1+\varepsilon^2}\Re\Big\{\varepsilon ^2 Z-\frac{1}{3}
Z^3,\, i(\varepsilon ^2Z+\frac{1}{3}Z^3),\,\varepsilon Z^2\Big\},
\quad Z\in \mathbb C,
\end{equation*}
 or equivalently, in the real form, by
\begin{align}\label{niza3}&\Psi_\varepsilon(X)= \frac{1}{1+\varepsilon^2}\big(\psi_1(X),\psi_2(X),\psi_3(X)\big)=\\&\frac{1}{1+\varepsilon^2}\Big(\varepsilon^2 X_1
-\frac{1}{3} (X_1^3-3X_1X_2^2),\, -\varepsilon
^2X_2+\frac{1}{3}(X_2^3-3X_1^2X_2),\,\varepsilon X_1^2-\varepsilon
X_2^2\Big). \notag
\end{align}
Let $\mathfrak E_\varepsilon =\Psi_\varepsilon (D_1)$. Then equality holds in \eqref{eqq}, since $\mathfrak E_\varepsilon$ is a minimal surface.
\paragraph{Points of Self-intersection on ${\boldsymbol{\mathfrak E}_\varepsilon}$.}
 Veli\u{c}kovi\'c \cite{velickovic}  observed that
  in polar coordinates $X = r(\cos \phi, \sin\phi)$,
\begin{align*}\psi_1(X) &= \varepsilon^2 r\cos \phi -(r^3/3) \cos (3\phi), \\  \psi_2(X) &= -\varepsilon^2 r\sin \phi -(r^3/3) \sin (3\phi),\\
 \psi_3(X)&=\varepsilon r^2 \cos(2\phi),
\end{align*}
and hence
\begin{equation}\psi_1(X)^2 + \psi_2(X)^2 + \frac 43\psi_3(X)^2 = \left(\frac{|X|^3}{3} + \varepsilon |X|\right)^2, \quad X \in D_1.
\label{polar4}\end{equation} Therefore, at a point of
self-intersection of $\mathfrak E_\varepsilon$,
$\Psi_\varepsilon(\tilde X) = \Psi_\varepsilon(\hat X)$ say,
\eqref{polar4} implies  that $\hat r = \tilde r$ since $r \mapsto
(r^3/3+\varepsilon^2 r)$ is monotone. If $r>0$ denote the common
value of $\hat r$ and $\tilde r$ at this point of intersection, the
following equations  must be satisfied:
\begin{subequations}\label{qolar}\begin{align}\label{qolar1}
\varepsilon^2 \cos \hat\phi -(r^2/3) \cos (3\hat\phi) &= \varepsilon^2 \cos \tilde\phi -(r^2/3) \cos (3\tilde\phi), \\
\label{qolar2}  \varepsilon^2 \sin \hat\phi +(r^2/3) \sin (3\hat\phi) &= \varepsilon^2 r\sin \tilde\phi +(r^2/3) \sin (3\tilde\phi),\\
\label{qolar3} \cos(2\hat\phi)&= \cos(2\tilde \phi), \quad
\hat\phi,\tilde \phi \in [0,2\pi),~  \hat\phi \neq \tilde\phi .
\end{align}
\end{subequations}
Clearly, \eqref{qolar3} implies    $\cos^2\hat \phi = \cos^
2\tilde\phi$, equivalently, $\cos\hat \phi = \pm\cos \tilde\phi$. If
$\cos \hat \phi = \pm \cos \tilde \phi = 0$,  it may be assumed that
$\hat \phi = 3\pi/2$ and $ \tilde \phi = \pi/2$. So \eqref{qolar1}
and \eqref{qolar3} are satisfied,  and \eqref{qolar2} holds if an
only if $\varepsilon^2 = r^2/3$. On the other hand, if $\cos \hat
\phi = \pm \cos \tilde \phi = 1$ it may be assumed that $ \hat \phi
= \pi$ and $ \tilde \phi = 0$ and \eqref{qolar1}-\eqref{qolar3} hold
if and only if $\varepsilon^2 = r^2/3$.

Of the remaining cases  first note that  $\cos \hat \phi =  \cos
\tilde \phi$  implies   $\sin \hat \phi =-  \sin \tilde \phi$ and
hence, \eqref{qolar} is satisfied if and only if \eqref{qolar2}
holds, i.e. $\hat \phi$ and $\tilde \phi$ satisfy
$$
  \varepsilon^2  +r^2(1 -(4/3) \sin^2\phi) = 0,
$$
since $\sin (3\phi) =3 \sin \phi -4 \sin^3\phi$. Finally, when $\cos
\hat \phi =  -\cos \tilde \phi$ there are two possibilities, $\sin
\hat \phi =  \sin \tilde \phi$ and  $\sin \hat \phi =  -\sin \tilde
\phi$. In the first case, \eqref{qolar2} is satisfied if $\sin \hat
\phi =  \sin \tilde \phi$, and the system \eqref{qolar} is satisfied
if and only if
$$ \varepsilon^2 + r^2(1- (4/3) \cos^2\phi) =0.
$$
because since  $\cos (3\phi) = 4 \cos^3\phi - 3 \cos \phi$. In the
final case, when $\cos \hat \phi =  -\cos \tilde \phi$ and $\sin
\hat \phi =  -\sin \tilde \phi$, \eqref{qolar} is satisfied if  and
only if
$$
 \varepsilon^2  +r^2(1 -(4/3) \sin^2\phi) = 0 \text{ and } \varepsilon^2 + r^2(1- (4/3) \cos^2\phi) =0,
$$
and in particular only if $\cos^2\phi = \sin^2\phi$.

Note that in all cases $|X|^2 \geq 3\varepsilon^2$, if
$\Psi_\varepsilon(X)$ is a self-intersection point of $\mathfrak
E_\varepsilon$, and  self-intersection curves on $\mathfrak
E_\varepsilon$ approach the origin in $\RR^3$ as $\varepsilon \to
0$.

\paragraph{Tangent Vectors to $\mathfrak E_\varepsilon$.}
Obviously when $\varepsilon \neq 0$ the vectors
\begin{equation}\label{niza4}\begin{split}
    (1+\varepsilon^2)\partial_1\Psi_\varepsilon(X)&=(\varepsilon^2-(X_1^2-X_2^2),\, -2
    X_1X_2,\, 2\varepsilon X_1)=:\mathbf a_{\varepsilon}(X),\\
(1+\varepsilon^2)\partial_2\Psi_\varepsilon(X)&=(2X_1X_2,\,
-\varepsilon^2-(X_1^2-X_2^2),\, -2\varepsilon X_2)=:\mathbf
b_{\varepsilon}(X)
\end{split}\end{equation} are non-zero and tangent to $\mathfrak E_\varepsilon$.
They are  orthogonal to one another because
\begin{multline*}
\mathbf a_{\varepsilon}(X)\cdot\mathbf b_{\varepsilon}(X)=
2\varepsilon^2X_1X_2-2 X_1^3X_2+2 X_1X_2^3\\+ 2\varepsilon^2
X_1X_2+2X_1^3X_2-2X_1X_2^3-4\varepsilon^2X_1X_2=0.
\end{multline*}
Now with
\begin{equation}\label{niza5}
    \lambda_{\varepsilon}(X):= \varepsilon^2+ X_1^2+X_2^2= \varepsilon^2+|X|^2,
\end{equation}
\begin{multline*}
|\mathbf a_{\varepsilon}(X)|^2=
\varepsilon^4-2\varepsilon^2(X_1^2-X_2^2)+(X_1^2-X_2^2)^2\\
+4X_1^2 X_2^2+4\varepsilon^2
X_1^2=(\varepsilon^2+X_1^2+X_2^2)^2=\lambda_{\varepsilon}^2(X).
\end{multline*}
and
\begin{multline*}
|\mathbf b_{\varepsilon}(X)|^2= 4X_1^2X_2^2 +\varepsilon ^4
+2\varepsilon^2(X_1^2-X_2^2)+(X_1^2-X_2^2)^2\\
+4\varepsilon^2
X_2^2=(\varepsilon^2+X_1^2+X_2^2)^2=\lambda_{\varepsilon}^2(X),
\end{multline*}
which in summary says that
\begin{equation}\label{niza6}
|\mathbf a_{\varepsilon}(X)|=|\mathbf
b_{\varepsilon}(X)|=\lambda_{\varepsilon}(X), \quad \mathbf
a_{\varepsilon}(X)\cdot\mathbf b_{\varepsilon}(X)=0,\quad X \in
\RR^2.
\end{equation}
Since the coefficients $g_{ij}$ of the first fundamental form of the
immersion $\Psi_\varepsilon$ are
$$
g_{11}=\partial_1\Psi_\varepsilon\cdot\partial_1\Psi_\varepsilon,
\quad
g_{22}=\partial_2\Psi_\varepsilon\cdot\partial_2\Psi_\varepsilon,
\quad
g_{12}=g_{21}=\partial_1\Psi_\varepsilon\cdot\partial_2\Psi_\varepsilon,
$$
it follows by \eqref{niza4} and \eqref{niza6} that
$$
g_{11}=g_{22}=\frac{\lambda_{\varepsilon}^2}{(1+\varepsilon^2)^2}
\text{ and } g_{12}=g_{21}=0.
$$
As  in the
Introduction,  the immersion $\Psi_\varepsilon:D_1\to \mathbb R^n$
has a conformal metric $g_{ij}=e^{2f_{\varepsilon}}\delta_{ij}$ and the conformal factor $\exp{(f_{\varepsilon}(X))}$ at
$\Psi_\varepsilon(X)$ is given by
\begin{equation}\label{niza7}
  f_{\varepsilon}(X)= \log \lambda_{\varepsilon}(X) -\log(1+\varepsilon^2), ~  X \in D_1, \quad  f_{\varepsilon}(\partial D_1) =0.
\end{equation}

\paragraph{The Normal to $\mathfrak E_\varepsilon$.} A unit normal
$\mathbf n_{\varepsilon}(X)$ to $\mathfrak E_\varepsilon$ at
$\Psi_\varepsilon(X)$ is given by
\begin{equation*}\label{niza8}
    \mathbf n_{\varepsilon}(X) =\frac{\mathbf
    a_{\varepsilon}(X)\times\mathbf b_{\varepsilon}(X)}{|\mathbf a_{\varepsilon}(X)\times\mathbf b_{\varepsilon}(X)|},
\end{equation*}
where
$$
\mathbf
    a_{\varepsilon}(X)\times\mathbf b_{\varepsilon}(X)=\left|
\begin{array}{ccc}
\displaystyle { \boldsymbol i} &\displaystyle{\boldsymbol j} & \displaystyle{\boldsymbol k} \\
\displaystyle { \varepsilon^2-(X_1^2-X_2^2)} &\displaystyle{
-2X_1X_2} &\displaystyle{2\varepsilon X_1}
\\ \displaystyle {2X_1X_2} &
\displaystyle{-\varepsilon^2-(X_1^2-X_2^2)}&\displaystyle{-2\varepsilon
X_2
 }
  \end{array}
\right|,
$$
and the $\bs i$ component of which is
$$
4\varepsilon X_1X_2^2+2\varepsilon^2 X_1-2\varepsilon
X_1X_2^2+2\varepsilon X_1^3 =2\varepsilon^3 X_1+2\varepsilon X_1
(X_1^2+X_2^1)=2\lambda_{\varepsilon}(X) \varepsilon X_1,
$$
the $\bs j$ component is
$$
4\varepsilon X_1^2 X_2+2\varepsilon^3 X_2 -2\varepsilon
X_2(X_1^2-X_2^2) =(\varepsilon^2+X_1^2+X_2^2)2\varepsilon X_2=
2\lambda_{\varepsilon}(X) \varepsilon X_2,
$$
and the $\bs k$ component is
$$
-(\varepsilon^4-(X_1^2-X_2^2)^2)+4 X_1^2 X_2^2
=-(\varepsilon^4-(X_1^1+X_2^2)^2)= \lambda_{\varepsilon}(X)
(X_1^2+X_2^2-\varepsilon^2).
$$
Thus
$$
\mathbf a_{\varepsilon}(X)\times \mathbf b_{\varepsilon}(X)
=\lambda_{\varepsilon}(X) \mathbf c_{\varepsilon}(X),
$$
where $\lambda_{\varepsilon}>0$ is defined by \eqref{niza5},
\begin{equation}\label{niza9}
    \mathbf c_{\varepsilon}(X)=\big(\, 2\varepsilon X_1,\,\, \, 2\varepsilon X_2,
    \,\,\,
    |X|^2-\varepsilon^2\, \big),
\end{equation}
\begin{align}\notag
|\mathbf a_{\varepsilon}(X)\times\mathbf
b_{\varepsilon}(X)|^2&=\lambda_{\varepsilon}^2(4\varepsilon |X|^2
+\varepsilon^4-2\varepsilon^2|X|^2+|X|^4)\\&=
\lambda_{\varepsilon}(X)^2(\varepsilon^2+|X|)^2=\lambda_{\varepsilon}(X)^4,\notag
\intertext{and so} \label{niza10}
     \mathbf n_{\varepsilon}(X) &=\frac{\mathbf a_{\varepsilon}(X)\times\mathbf b_{\varepsilon}(X)}{|\mathbf a_{\varepsilon}(X)\times\mathbf b_{\varepsilon}(X)|}\, =\frac{ \mathbf c_{\varepsilon}(X)}{\lambda_{\varepsilon}(X)}.
\end{align}
\begin{remark}\label{nizaremark}
Note that $\nn_\varepsilon$,  which is the  Gauss  map of the
surface
 $\mathfrak E_\varepsilon$ parameterized by \eqref{niza3},  is smooth and  injective.  The smoothness is obvious, and if  $X,X' \in D_1$ are
such that $\mathbf n_{\varepsilon}(X')=\mathbf n_{\varepsilon}(X)$
then, by \eqref{niza9},
$$
\frac{X'}{\lambda_{\varepsilon}(X')}=\frac{X}{\lambda_{\varepsilon}(X)}
\quad \text{ and }\quad
\frac{|X'|^2-\varepsilon^2}{|X'|^2+\varepsilon^2}=
\frac{|X|^2-\varepsilon^2}{|X|^2+\varepsilon^2}
$$
From the first,  $X'=kX$ for some $k>0$ and then, from the second,
$k=1$. \qed\end{remark}

\subsection{Remark on Theorem \ref{aelita22}}\label{r1}
To show  the hypotheses of  Theorem \ref{aelita22}  are sharp the
first step is to calculate the norm of the conformal factor and
estimate the norm of  $\Phi_{\varepsilon}$. Since, by \eqref{niza7},
$$
\nabla
f_{\varepsilon}(X)=\frac{1}{\lambda_{\varepsilon}(X)}(2X_1,\,\,
2X_2)\quad \text{ and } \quad |\nabla f_{\varepsilon}(X)|^2=
\frac{4|X|^2}{\lambda_{\varepsilon}(X)^2},
$$
it follows that
\begin{align}\notag
\int_{D_1} |\nabla f_{\varepsilon}|^2\, dX&=4\int_{D_1}\frac{|X|^2\,
dX}{\varepsilon^2+|X|^2}\\\notag &=
4\int_0^{2\pi}\int_0^1\frac{r^2\, rdr}{(\varepsilon^2+r^2)^2}d\theta
= 4\pi\int_0^1\frac{s ds}{(s+\varepsilon^2)^2}\\&\notag =
4\pi\left\{\log\Big(\frac{1}{\varepsilon^2}+1\Big)-\Big(\frac{1}{1+\varepsilon^2}\Big)\right\},\\
\label{niza14} &=\Delta(\varepsilon)^2, \text{ where }
\Delta(\varepsilon)^2=4\pi\left(\log
\frac{1}{\varepsilon^2}-1\right) +O(\varepsilon^2),
\end{align}
as  $\varepsilon \to 0$. Hence
\begin{equation}\label{niza15}
    \|\nabla f_{\varepsilon}\|_{L^2(D_1)}\to \infty\text{~as~}\varepsilon\to 0.
\end{equation}
To investigate $\Phi_{\varepsilon}$ as $\varepsilon\to 0$ note from
 \eqref{niza12} that
\begin{equation}\label{niza16}\begin{split}\text{meas}(\mathfrak E_\varepsilon)&= \int_{D_1}|\Phi_{\varepsilon}|\,
dX=4\varepsilon^2\int_{D_1}\frac{dX}{(\varepsilon^2+|X|^2)^2}
\\&=4\varepsilon^2\int_0^{2\pi}\int_0^1\frac{rdr}{(\varepsilon^2+r^2)^2}
=\frac{4\pi}{1+\varepsilon^2} \to 4\pi ~\text{ as $\varepsilon\to
0$.}
\end{split}\end{equation}
To estimate $\|\nn\|_{L_2(D_1)}$, note from \eqref{niza9} and
\eqref{niza10} that
\begin{multline*}
\partial_1 \mathbf n_{\varepsilon} (X)=\frac{1}{\lambda_{\varepsilon}(X)^2}\Big(-4\varepsilon X_1^2,
-4\varepsilon X_1X_2, -2X_1|X|^2+2X_1 \varepsilon^2\Big)\\
+\frac{1}{\lambda_{\varepsilon}(X)^2}\Big(2\varepsilon
(\varepsilon^2+X_1^2+X_2^2), 0,
2X_1\varepsilon^2+2X_1(X_1^2+X_2^2)\Big)\\=
\frac{1}{\lambda_{\varepsilon}(X)^2}\Big(2\varepsilon^3-2\varepsilon
(X_1^2-X_2^2),\,\, -4\varepsilon X_1X_2,\,\, 4\varepsilon^2
X_1\Big),
\end{multline*}
and
\begin{multline*}
|\partial_1\mathbf
n_{\varepsilon}|^2=\frac{1}{\lambda_{\varepsilon}(X)^4}\Big(\,4\varepsilon^6-8\varepsilon^4(X_1^2-X_2^2)+
\\4\varepsilon^2(X_1^4-2X_1^2X_2^2+X_2^4)+16 \varepsilon^4
X_1^2+16\varepsilon^2 X_1^2X_2^2\,\Big)
\\
=\frac{1}{\lambda_{\varepsilon}(X)^4}\Big(\,4\varepsilon^6+8\varepsilon^4|X|^2+4\varepsilon^2|X|^4\,\Big)
=\frac{4\varepsilon^2\lambda_{\varepsilon}(X)^2}{\lambda_{\varepsilon}(X)^4}=\frac{4\varepsilon^2}{\lambda_{\varepsilon}(X)^2}.
\label{niza19a}
\end{multline*}
A similar calculation yields
\begin{equation*}\label{niza19b}
|\partial_2\mathbf
n_{\varepsilon}(X)|^2=\frac{4\varepsilon^2}{\lambda_{\varepsilon}(X)^2},
\end{equation*}
and hence
\begin{equation}\label{niza18}\begin{split}
    \int_{D_1}|\nabla\mathbf n_{\varepsilon}(X)|^2\,
    dX&=8\int_{D_1}\frac{\varepsilon^2dX}{(\varepsilon^2+|X|^2)^2}=
    8\int_0^{2\pi}\int_0^1 \frac{\varepsilon^2
    rdr\,d\theta}{(\varepsilon^2+r^2)^2}\\&=  8\pi\varepsilon^2\int_0^1 \frac{
    ds}{(\varepsilon^2+s)^2}=
    8\pi\Big(\frac{1}{1+\varepsilon^2}\Big)\to
    8\pi\text{~as~}\varepsilon\to 0.
\end{split}\end{equation}
Finally, recall from \eqref{niza7} that $f_\varepsilon = 0 $ on
$\partial D_1$, whence $f_\varepsilon \in W^{1,2}_0(D_1)$, and
\begin{equation*}
 \int_{D_1}\Phi_{\varepsilon} g_{\varepsilon}\,dX = \|\nabla f_{\varepsilon}\|_{L^2(D_1)}  \text{ where } g_{\varepsilon} = f_{\varepsilon}/\|\nabla f_{\varepsilon}\|_{L^2(D_1)},
\end{equation*}
because $-\Delta f_{\varepsilon}=\Phi_{\varepsilon}$. Since
$\|\nabla g_{\varepsilon}\|_{L^2(D_1)} =1$, it is immediate from
\eqref{niza15} that
\begin{equation} \|\Phi_{\varepsilon}\|_{W^{-1,2}(D_1)}\to \infty \text{~as~}\varepsilon\to 0.\label{niza19}
\end{equation}
Since equality holds in \eqref{eqq}, by \eqref{niza16}, \eqref{niza18} and \eqref{niza19}, the hypotheses
of Theorem \ref{aelita22} are sharp.

\subsection{Remark on  Theorem \ref{aelita30}}\label{r2}
A consequence of  \eqref{aelita32} in Theorem \ref{aelita30} is
that for every
\begin{equation}\label{niza21}
\zeta\in W^{1,2}_0(D_1),\quad \|\nabla \zeta\|_{L^2(D_1)}=1,
\end{equation}
for every $\varepsilon >0$ and for every $\mathcal A\subset \mathbb
S^2$ with  measure $\mu>0$,
\begin{equation}\label{niza17}
    \Big|\int_{D_1} \Phi_{\varepsilon} \zeta\, dX-\frac{4\pi}{\mu}\int_{\mathcal
    F} \Phi_{\varepsilon}\zeta\, dX\Big|\leq c(\mu),
\end{equation}
 even if
\begin{equation}\label{niza20}
    \Big|\int_{D_1} \Phi_{\varepsilon} \zeta\, dX\Big|\to \infty \text{~as~} \varepsilon \to
    0.
\end{equation}
For the significance of this observation, recall Remark \ref{holo}.
Here this is illustrated by showing that $ \zeta_{\varepsilon}
:=g_\varepsilon$ defined above  has  the property that, for any
$\mathcal A$ with positive measure,  \eqref{niza21}, \eqref{niza17}
and \eqref{niza20}  hold with $\mathcal
F_\varepsilon=\nn_\varepsilon^{-1}(\mathcal A)$.

Recall from the argument for \eqref{niza19}  that with
$g_{\varepsilon} = f_{\varepsilon}/\|\nabla
f_{\varepsilon}\|_{L^2(D_1)} \in W^{1,2}_0$,
\begin{equation}\label{niza21a}
 \int_{D_1}\Phi_{\varepsilon} g_{\varepsilon}\,dX = \|\nabla f_{\varepsilon}\|_{L^2(D_1)},
\end{equation}
So let $\zeta_\varepsilon = g_\varepsilon $. Then, from
\eqref{niza7},
\begin{equation*}
    \zeta_{\varepsilon}=\frac{f_{\varepsilon}}{\|\nabla f_{\varepsilon}\|_{L^2(D_1)}}=\frac{\log
    \lambda_{\varepsilon}-\log(1+\varepsilon^2)}{\Delta(\varepsilon)}
\end{equation*}
where $\Delta(\varepsilon)$ is defined by \eqref{niza14}. It follows
from \eqref{niza21a}, \eqref{niza14} and \eqref{niza15} that
$\zeta_{\varepsilon}$ satisfies conditions \eqref{niza21} and
\eqref{niza20}, and, from \eqref{niza12},
\begin{equation*}\label{niza22}
    \Phi_{\varepsilon}\,\zeta_{\varepsilon}=-\frac{4\varepsilon^2}{\lambda_{\varepsilon}^2\Delta(\varepsilon)}(\log
    \lambda_{\varepsilon}-\log(1+\varepsilon^2)).
\end{equation*}
To proceed,  is convenient to introduce new independent variables,
$$
X=\varepsilon \xi, \quad \xi=(\xi_1, \xi_2),
$$
so that $ \lambda_{\varepsilon}=\varepsilon^2(1+|\xi|^2)$ and $\log
\lambda_{\varepsilon}=\log\varepsilon^2+\log(1+|\xi|^2), $ and in
these variables
\begin{equation}\label{niza23}
    \Phi_{\varepsilon}\zeta_{\varepsilon}=\frac{-4\log\varepsilon^2}{\Delta(\varepsilon)\varepsilon^2(1+|\xi|^2)^{2}}
   + \frac{1}{\varepsilon^2}\mathbf I,
\end{equation}
where
\begin{equation}\label{niza24}
    \mathbf
    I(\xi)=4\frac{\log(1+\varepsilon^2)-\log(1+|\xi|^2)}{\Delta(\varepsilon)(1+|\xi|^2)^2}.
\end{equation}
By Remark \ref{nizaremark},  $X\to \mathbf n_{\varepsilon}(X)$ is
injective and its Jacobian is $|\Phi_{\varepsilon}|$. For $\mathcal
A\subset \mathbb S^2$ let $\mathcal
F_\varepsilon=\nn_{\varepsilon}^{-1}(\mathcal A)$ and
$\mu=\text{meas}\,(\mathcal A)$. Then, from \eqref{niza12},
\begin{equation}\label{niza25}
    \mu=\int_{\mathcal F_\varepsilon} |\Phi| dX =4\varepsilon^2\int_{\mathcal
    F_\varepsilon}\frac{dX}{(\varepsilon^2+|X|^2)^2}=4\int_{\mathcal
    F_\varepsilon/\varepsilon}\frac{d\xi}{(1+|\xi|^2)^2}.
\end{equation}
Note also that
\begin{equation}\label{niza26}
4\int_{\mathbb R^2}\frac{d\xi}{(1+|\xi|^2)^2} = 4\pi \int_0^\infty
\frac{dt}{(1+t)^2}=4\pi.
\end{equation}
Therefore, by \eqref{niza23}  and \eqref{niza25},
\begin{align}
\notag &\int_{D_1}\Phi_{\varepsilon} \zeta_{\varepsilon}\,
dX-\frac{4\pi}{\mu}\int_{\mathcal F_\varepsilon}\Phi_{\varepsilon}
\zeta_{\varepsilon}\, dX\\&\notag =\frac{-\log
\varepsilon^2}{\Delta(\varepsilon)}
\Big\{\int_{D_{1/\varepsilon}}\frac{4d\xi}{(1+|\xi|^2)^2}-\frac{4\pi}{\mu}
\int_{\mathcal
F_\varepsilon/\varepsilon}\frac{4d\xi}{(1+|\xi|^2)^2}\Big\}
+\Big\{\int_{D_{1/\varepsilon}}\mathbf I\,
d\xi-\int_{\mathcal F_\varepsilon/\varepsilon}\mathbf I\, d\xi\Big\}\\
\notag &= \frac{-\log \varepsilon^2}{\Delta(\varepsilon)}
\Big\{\int_{D_{1/\varepsilon}}\frac{4d\xi}{(1+|\xi|^2)^2}-4\pi
\}\Big\} +\Big\{\int_{D_{1/\varepsilon}}\mathbf I\,
d\xi-\int_{\mathcal F_\varepsilon/\varepsilon}\mathbf I\, d\xi\Big\}
\\&=:J_1+J_2. \notag
\end{align}
Now by \eqref{niza14}, \eqref{niza25} and \eqref{niza26},
\begin{align}\notag
J_1&=\frac{-\log \varepsilon^2}{\Delta(\varepsilon)}
\Big\{\int_{D_{1/\varepsilon}}\frac{4d\xi}{(1+|\xi|^2)^2}-4\pi\Big\}=
\frac{4\log \varepsilon^2}{\Delta(\varepsilon)} \,\,\int_{|\xi|\geq
\varepsilon^{-1}}\frac{d\xi}{(1+|\xi|^2)^2}\\
&= \frac{4\pi\varepsilon \log
\varepsilon^2}{(1+\varepsilon)\Delta(\varepsilon)} \leq
c\varepsilon\sqrt{|\log\varepsilon|} \text{ as } \varepsilon \to 0.
\label{niza29}\end{align} Next,  by \eqref{niza24},
\begin{equation}\label{niza30}\begin{split}
|J_2|\leq \frac{c}{\Delta(\varepsilon)}\int_{\mathbb R^2}
\frac{\log(1+|\xi|^2)+1}{(1+|\xi|^2)^2}\,d\xi\leq
\frac{c}{\Delta(\varepsilon)}\leq
\frac{c}{\sqrt{|\log\varepsilon|}}.
\end{split}\end{equation}
Combining \eqref{niza29} and \eqref{niza30}  leads to
$$
 \Big|\int_{D_1} \Phi_{\varepsilon} \zeta_{\varepsilon}\, dX-\frac{4\pi}{\mu}\int_{\mathcal
    F} \Phi_{\varepsilon}\zeta_{\varepsilon}\, dX\Big|\leq \frac{c}{\sqrt{|\log\varepsilon|}},
$$
which yields desired estimate \eqref{niza17}, while \eqref{niza21}
and \eqref{niza20} are satisfied.

\subsection{Stereographic Projections}\label{duality}
Recall the standard stereographic projections,
$$ (X_1,X_2) \mapsto \frac{1}{|X|^2 +4} \Big( 4X_1,\quad 4X_2, \quad \pm(|X|^2-4)\Big),\quad X \in \RR^2,
$$
each of which maps the horizontal plane onto the unit sphere with
one of the poles, $+\boldsymbol k$ or $-\boldsymbol k$, removed.
From them, two families of immersions of $D_1$ into $\RR^3$ are
defined by  replacing $X$ by  $ 2X/\varepsilon$ to obtain
\begin{equation}\label{niza31}\begin{split}
\Psi^{\varepsilon}_{\pm}(X)&=\frac{1}{\lambda_{\varepsilon}(X)}
\left ( {2\varepsilon X_1}, { 2\varepsilon X_2},\pm(
{|X|^2-\varepsilon^2} \right),\\\lambda_{\varepsilon}(X) &=
\varepsilon ^2 + |X|^2,\quad \,X \in D_1.
\end{split}\end{equation}
Both $\Psi_\pm^{\varepsilon}$ map $D_1$ into $\SS$ leaving
 a small round hole, centred  at $\pm\boldsymbol k$  with diameter that  vanishes as $\varepsilon\to 0$. Since  $\Psi^{\varepsilon}_\pm$ coincides with its  Gauss map,
\begin{equation*}\label{niza32}
    \mathbf n^{\varepsilon}_{\pm}(X)=\Psi^{\varepsilon}_{\pm}(X), \quad X\in D_1,
\end{equation*}
it is immediate from \eqref{niza5}, \eqref{niza9}, \eqref{niza10}
and \eqref{niza31} that
$$\nn_\varepsilon(X) = \nn^\varepsilon_+(X),\quad X \in D_1.
$$
In other words,  the Gauss maps of the Enneper surface parameterized
by \eqref{niza3}, and  of the stereographic projections
$\Psi^\varepsilon_+$, coincide. ($\Psi^1_-$ was considered without
reference to Enneper surfaces in \cite[Example 5.2.2]{helein}.)

Recall from  \eqref{aelita19}, \eqref{niza10} that
\begin{equation}\label{niza11}\begin{split}
    \Phi_{\varepsilon}(X) &= \mathbf n_{\varepsilon} (X)\cdot(\partial_1\mathbf n_{\varepsilon} (X)\times
    \partial_2\mathbf n_{\varepsilon} (X))\\&=\frac{1}{\lambda_{\varepsilon}(X)^3}\,\mathbf c_{\varepsilon}(X)\cdot(\partial_1\mathbf c_{\varepsilon}(X) \times
    \partial_2\mathbf c_{\varepsilon}(X)),\end{split}
\end{equation}
where
$$
\mathbf
    \partial_1\mathbf c_{\varepsilon}(X)\times\partial_2\mathbf c_{\varepsilon}(X)=\left|
\begin{array}{ccc}
\displaystyle { \bs i} &\displaystyle{\bs j} & \displaystyle{\bs k} \\
\displaystyle { 2\varepsilon} &\displaystyle{0} &\displaystyle{2
X_1}
\\ \displaystyle {0} &
\displaystyle{2\varepsilon}&\displaystyle{2 X_2
 }
  \end{array}
\right|=-4\varepsilon X_1\bs i -4\varepsilon X_2 \bs j+
4\varepsilon^2\bs k.
$$
and hence
$$
\mathbf c_{\varepsilon}(X)\cdot(\partial_1\mathbf c_{\varepsilon}(X)
\times
    \partial_2\mathbf c_{\varepsilon}(X))=-8\varepsilon^2
    |X|^2+4\varepsilon^2|X|^2-4\varepsilon^4=-4\varepsilon^2 \lambda_{\varepsilon}(X).
$$
From this and \eqref{niza11},
\begin{equation}\label{niza12}
    \Phi_{\varepsilon}(X)=\mathbf n_{\varepsilon}(X)\cdot(\partial_1\mathbf n_{\varepsilon}(X)\times
    \partial_2\mathbf n_{\varepsilon}(X))=-\frac{4\varepsilon^2}{\lambda_{\varepsilon}(X)^2}.
\end{equation}

\begin{remark}
At first glance,  formula \eqref{niza12} is misleading since it
appears to imply that the curvature of the convex sphere $\mathbb S^2$, which
is proportional  to $\Phi$, is negative. However,
the immersions $\Psi^-$ and $\Psi^+$ have the different
orientations. The orientation of $\Psi^-$ is positive which means
  $\Phi=4\varepsilon^2/ \lambda^2$, corresponding to the
the curvature of the sphere being positive.  By  contrast,
the orientation of the immersion $\Psi^+$ is negative and $\Phi$
changes the sign.  In fact, $\Phi=-Ke^{2f}$ in formula \eqref{niza12}.
Since the orientation of the
Enneper  immersion $\Psi_\varepsilon$ defined by \eqref{niza3} is
positive, and the corresponding function $\Phi_\varepsilon$ is
negative.  This means that the Enneper surface has the negative
curvature, and so  is hyperbolic. \qed
\end{remark}

\ed